\documentclass[12pt]{amsart}

\usepackage{amsfonts,amsmath,latexsym,amssymb,verbatim,amsbsy,times}
\usepackage{amsthm}%, pdfsync}
\usepackage{pstricks}
\usepackage{graphicx}
\usepackage{caption}
\usepackage[top=0.8in, bottom=0.7in, left=0.8in, right=0.8in]{geometry}

\usepackage[colorlinks=true, pdfstartview=FitV, linkcolor=blue,citecolor=green, urlcolor=blue]{hyperref}

\usepackage{enumitem}

\theoremstyle{plain}
\newtheorem{THEOREM}{Theorem}[section]

\newtheorem{LEMMA}[THEOREM]{Lemma}

\newtheorem{QUESTION}[THEOREM]{Question}

\newtheorem{corollary}[THEOREM]{Corollary}
\newtheorem{lemma}[THEOREM]{Lemma}
\newtheorem{proposition}[THEOREM]{Proposition}

\theoremstyle{definition}

\theoremstyle{remark}
\newtheorem{REMARK}[THEOREM]{Remark}

\newtheorem{claim}[THEOREM]{Claim}

\newcommand{\thm}[1]{Theorem~\ref{#1}}
\newcommand{\lem}[1]{Lemma~\ref{#1}}

\newcommand{\prop}[1]{Proposition~\ref{#1}}

\newcommand{\N}{\ensuremath{\mathbb{N}}}   %%% naturals
\newcommand{\R}{\ensuremath{\mathbb{R}}}   %%% reals
\def \a {\alpha}
\def \b {\beta}
\def \d {\delta}

\def \e {\epsilon}

\def \l {\lambda}

\def \n {\nabla}
\def \s {\sigma}
\def \th {\theta}

\def \w {\omega}

\def \O {\Omega}

\def \cE {\mathcal{E}}

\def \cH {\mathcal{H}}

\def \cM {\mathcal{M}}

\def \< {\langle}
\def \> {\rangle}
\def \p {\partial}

\def \ss {\subset}

\DeclareMathOperator{\supp}{supp} %
\newcommand{\rest}[2]{#1\raisebox{-0.3ex}{\mbox{$\mid_{#2}$}}}
\newcommand{\dist}[2]{\mathrm{dist}(#1,#2)}

\def \Lip {\mathrm{Lip}}

\def \dt  {\, \mbox{d}t}
\def \dtau  {\, \mbox{d}\tau}
\def \dx  {\, \mbox{d}x}
\def \dy  {\, \mbox{d}y}

\def \dl  {\, \mbox{d}\l}
\def \dmu  {\, \mbox{d}\mu}

\def \dth  {\, \mbox{d}\th}
\def \dE  {\, \mbox{d}\cE}

\def \OR {\mathcal{OR}}

\newcommand{\mres}{\mathbin{\vrule height 1.6ex depth 0pt width
0.13ex\vrule height 0.13ex depth 0pt width 1.3ex}}

\begin{document}

\title{The Energy Measure for the Euler and Navier-Stokes Equations}
\author{Trevor M. Leslie and Roman Shvydkoy}

\email{tlesli2@uic.edu; shvydkoy@uic.edu}

\address{Department of Mathematics, Statistics, and Computer Science \\851 S Morgan St, M/C 249 \\ University of Illinois at Chicago, Chicago, IL, 60607}

\begin{abstract}
The potential failure of energy equality for a solution $u$ of the Euler or Navier-Stokes equations can be quantified using a so-called `energy measure': the weak-$*$ limit of the measures $|u(t)|^2\dx$ as $t$ approaches the first possible blowup time. We show that membership of $u$ in certain (weak or strong) $L^q L^p$ classes  gives a uniform lower bound on the lower local dimension of $\cE$; more precisely, it implies uniform boundedness of a certain upper $s$-density of $\cE$. We also define and give lower bounds on the `concentration dimension' associated to $\cE$, which is the Hausdorff dimension of the smallest set on which energy can concentrate.  Both the lower local dimension and the concentration dimension of $\cE$ measure the departure from energy equality.  As an application of our estimates, we prove that any solution to the $3$-dimensional Navier-Stokes Equations which is Type-I in time must satisfy the energy equality at the first blowup time. %Furthermore, we give new criteria for energy conservation (equality) in terms of the dimension of the singularity set and classical $L^q L^p$ conditions.
\end{abstract}

\maketitle

\section{Introduction}

We consider the incompressible Euler or Navier-Stokes initial value problem on $\R^n$:
\begin{align}
\label{e:moment}
\p_t u + u\cdot \n u - \nu \Delta u & = -\n p, \\
\label{e:incomp}
\n \cdot u & = 0, \\
\label{e:IC}
u(t_0) & = u_0.
\end{align}
Here we understand that $\nu = 0$ and $n\ge 3$ if we are considering the Euler Equations, and that $n=3$ if $\nu > 0$.  In either case, we assume that $u_0\in H^{\frac n2 + 1 + \e}(\R^n)$ for some $\e>0$, so that there exists a unique local-in-time solution 
\begin{equation}
\label{e:EEreg}
u\in C([t_0,t_1); H^{\frac{n}{2} + 1 + \e}(\R^n)),
\end{equation}
for some $t_1>0$, with the associated pressure given by 
\begin{equation}
\label{e:pdef}
p = R_i R_j (u_i u_j),
\end{equation}
where $R_i$, $R_j$ denote the classical Riesz transforms.
%\begin{equation}
%\label{e:pdef}
%p(x,t) = -\frac{|u(x,t)|^2}{n} + p.v.\int_{\R^n} %K_{ij}(x-y)u_i(y)u_j(y)\dy.
%\end{equation}
%Here
%\[
% K_{ij}(y) = \frac{n y_i y_j - \d_{ij}|y|^2}{n\w_n |y|^{n+2}},
%\]
%where $\w_n$ is the volume of the unit ball in $\R^n$. 
We assume that $(u,p)$ can be extended to some larger time interval $[t_0, T]$, with $T\ge t_1$, with $u$ a weak solution on the larger interval, which is weakly continuous in $L^2$ at $t=t_1$.  If $\nu>0$, we assume that $u$ is a Leray-Hopf solution on $[t_0,T]$. Let $\O\subset \R^n$ be a bounded subdomain and assume $u\in L^3(t_0,t_1; L^3(\O))$ (this is automatic if $\nu>0$).  We will work either on $\O$ or on the full space $\R^n$; in the latter case we will assume without further comment that $u\in L^3(t_0, t_1; L^3(\R^n))$.  We stress that even when we work on $\O$, the pair $(u,p)$ will solve \eqref{e:moment}--\eqref{e:IC} on the full space.  

Trivial manipulations of \eqref{e:moment} and \eqref{e:incomp} yield the following local energy equality for all nonnegative $\s\in C_0^\infty(\O\times [t_0,t_1])$ and all $t\in [t_0,t_1)$:
\begin{equation}
\label{e:localee}
\begin{split}
\int_\O |u(t)|^2 \s(t) \dx =
& \int_\O |u(t_0)|^2 \s(t_0) \dx -2\nu\int_{t_0}^t \int_\O |\n u|^2 \s\dx\dt  \\
& + \int_{t_0}^t \int_\O |u|^2(\p_t \s + \nu \Delta \s)\dx\,\dtau + \int_{t_0}^t \int_\O (|u|^2 + 2p)u \cdot \n \s\dx \dtau.
\end{split}
\end{equation}

We are concerned with the question of whether \eqref{e:localee} continues to hold when $t=t_1$.  If $u$ remains regular at $t=t_1$, then the answer is clearly affirmative; therefore we assume without loss of generality that $u$ does in fact lose regularity at time $t=t_1$.  In this case we can legitimately claim only that the local energy \emph{inequality} holds at $t=t_1$ for all non-negative test-functions $\s$. This is a simple consequence of the weak lower semicontinuity of the $L^2$ norm and the regularity in time of $\s$:
\begin{equation}
\label{e:localei}
\begin{split}
\int_\O |u(t_1)|^2 \s(t_1) \dx & \le  \lim_{t\to t_1^-} \int_\O |u(t)|^2\s(t)\dx \\
& =\int_\O |u(t_0)|^2 \s(t_0) \dx -2\nu\int_{t_0}^{t_1} \int_\O |\n u|^2 \s\dx\dt  \\
& \quad + \int_{t_0}^{t_1} \int_\O |u|^2(\p_t \s + \nu \Delta \s)\dx\dtau + \int_{t_0}^{t_1} \int_\O (|u|^2 + 2p)u \cdot \n \s\dx \dtau.
\end{split}
\end{equation}
We ask, then: in what circumstances may we conclude that \eqref{e:localee} survives the first blowup time, i.e. \eqref{e:localee} rather than just \eqref{e:localei} holds at $t=t_1$?  

\subsection{Background on the Energy Equality}
To begin with, we give one sufficient condition for the energy equality \eqref{e:localee} to hold at time $t=t_1$, which gives a partial answer to the question above, and which we will use  extensively below.  For $U$ an open subset of $\R^n$ and $I$ a relatively open interval in $[t_0, T]$, define the ``Onsager regular'' function class $\OR(\R^n \times I)$ and its local-in-space version $\OR(U\times I)$ as follows:
\[
\OR(\R^n \times I) = \{  f \in L^3(\R^n\times I):  \lim_{y \to 0 } \frac{1}{|y|}  \int_I \int_{\R^n}  |f(x+y,t) - f(x,t)|^3 \dx\dt  = 0\}.
\]
\[
\OR(U \times I) = \{  f \in L^3(U\times I): \s f \in \OR(\R^n \times I), \text{ for all } \s \in C_0^\infty(U) \}.
\]
We sometimes omit parts of the notation for these spaces when there is no risk of sacrificing clarity.

The result of \cite{RSSingFlows} states that if $u$ is a weak solution to the Euler equations on $[t_0,T]$, and if $u\in \OR(\O\times [t_0,T])$, then $u$ satisfies \eqref{e:localee} for every $t\in [t_0,T]$.  (Actually, the theorem states something slightly stronger, but this formulation is sufficient for our purposes.)  The proof carries over for solutions to the Navier-Stokes equations without difficulty.  Therefore, when we make the additional regularity assumptions \eqref{e:EEreg}  the relevant sufficient condition for \eqref{e:localee} to survive the first blowup time $t=t_1$ is that $u\in \OR(\O\times [t_0, t_1])$.

The quoted result of \cite{RSSingFlows} is a local critical version of a long list of preceding sufficient conditions documented in the extensive body of literature on the so-called Onsager conjecture.  This conjecture, formulated in 1949 by Lars Onsager \cite{Onsager}, states that $1/3$ is a critical smoothness in the sense that solutions to the Euler equations of smoothness greater than $1/3$ must conserve energy, and that solutions with smoothness less than $1/3$ might not.  The positive direction of this conjecture was resolved in \cite{CET} by Constantin, E, and Titi and has been subsequently refined in, for example, Duchon, Robert \cite{DR},  and   Cheskidov, et al \cite{CCFS}.  The other direction of the conjecture is not as relevant for the present work; however, we mention that it has been recently resolved by Isett \cite{Isett}, following a series of breakthrough ideas originating in topology by De Lellis and Sz\'ekelyhidi. We do not attempt to give a detailed overview of this side of the subject, instead we refer the reader to \cite{DS-review} for an extensive survey.  

The question of energy equality has of course also been extensively studied for Leray-Hopf solutions of the 3-dimensional Navier-Stokes equations; we mention only a few results.  Lions \cite{Lions} and Ladyzhenskaja et al. \cite{Ladyzhenskaya} proved independently that such solutions satisfy the (global) energy equality under the additional assumption $u\in L^4(t_0, T;L^4)$ (see also \cite{Serrin}, \cite{Shinbrot} for improvements in higher spatial dimensions).  Actually, the $L^4 L^4$ criterion is recoverable from that of \cite{RSSingFlows} (and earlier results), since $L^4 L^4 \cap L^2 H^1\subset \OR$ by interpolation. Later, Kukavica \cite{Kukavica} proved sufficiency of the weaker but dimensionally equivalent criterion $p\in L^2(t_0, T;L^2)$.  In \cite{CFS}, energy equality was proven for $u\in L^3 D(A^{5/12})$ on a bounded domain; an extension to exterior domains was proved in \cite{FT}. (Here $A$ denotes the Stokes operator.) In \cite{SS-pressure}, Seregin and \v{S}ver\'ak have proven energy equality (regularity, in fact) for suitable weak solutions whose associated pressure is bounded from below in some sense; this paper makes use of the low-dimensionality of the singular set for suitable weak solutions that is guaranteed by the celebrated Caffarelli-Kohn-Nirenberg Theorem \cite{CKN}.  Work by the second author \cite{ShvydkoyGeometric} and more recently by both authors \cite{LS2016b} proves energy equality under assumptions on the size and/or structure of the singularity set in addition to the integrability of the solution.  This last work in particular also considers (among other situations) the energy equality at the first time of blowup (or equivalently, under the assumption that the singularity set is restricted to a single time-slice).  

In the present work, our philosophy will be similar in spirit to that of \cite{ShvydConc} and \cite{LS2016b}, in that we will impose integrability assumptions on our solution $u$ and consider only the first time of blowup.  Of all the hypotheses on $u$ that we consider, however, there is one that deserves special attention, namely the case where a solution $u$ of the Navier-Stokes equations undergoes Type-I in time blowup. By this we mean that 
\begin{equation}
\label{e:typeINSE}
\|u(t)\|_{L^\infty(\R^3)} \le \frac{C}{\sqrt{t_1-t}}, 
\end{equation}
for some constant $C>0$.  The Type-I assumption is of particular significance because of its invariance under the natural rescaling for the Navier-Stokes equations.  See \cite{SS} for a discussion and further references.  In fact, it is proved in \cite{SS} that axially symmetric solutions of the Navier-Stokes equations which experience Type-I in time blowup (and satisfy some natural technical assumptions) are regular, and therefore satisfy the energy equality.  

\subsection{Definition of the Energy Measure}
Henceforth we restrict attention to the question posed earlier, regarding energy equality at the first blowup time. For the rest of the paper, we set $t_0 = -1$, $t_1 = 0$ for convenience.  Actually, we refine our question somewhat:

\begin{QUESTION}	
Suppose $u$ satisfies \eqref{e:moment}--\eqref{e:pdef}.  Under what additional integrability assumptions on $u$ may we conclude that \eqref{e:localee} holds at $t=t_1=0$?  If we cannot prove \eqref{e:localee} for a given integrability assumption on $u$, how bad is the worst failure of \eqref{e:localee} that we cannot eliminate under that same assumption?
\end{QUESTION}

Note that the second part of this question presupposes that we can meaningfully and quantifiably distinguish between different instances of failure of \eqref{e:localee}.  The tool that we use to justify the tacit assumption in this question (and address the question itself) is the energy measure $\cE$, which we define to be the weak-$*$ limit of the measures $|u(t)|^2\dx\mres\O$ as $t\to 0^-$. (The symbol $\mres$ denotes restriction of a measure onto a given set.)  To see that $\cE$ is well-defined, note that $|u(t)|^2\dx$ is a bounded sequence of Radon measures, so that there exists a subsequence $|u(t_k)|^2\dx\mres\O$ which converges weak-$*$ to some Radon measure. Any two such measures agree as distributions by \eqref{e:localee}. Thus $\cE$ is uniquely determined as a linear functional on $C_0(\O)$, by density of $C_0^\infty(\O)$ in $C_0(\O)$.  

We can reinterpret \eqref{e:localei} as saying that $\dE \geq |u(0)|^2 \dx\mres\O$ in the sense of measures, with equality if and only if \eqref{e:localee} is valid at $t=0$.  This fact clarifies how properties of the energy measure may be used to examine the possible failure of energy equality.  In particular, we introduce the following two quantities, the lower local dimension $d(x,\cE)$ of $\cE$ at $x\in \O$, and the concentration dimension $D$ of $\cE$ in $\O$, defined respectively by
\begin{equation}
\label{e:localdim}
d(x,\cE) = \liminf_{r \to 0}  \frac{\ln \cE(B_r(x))}{\ln r},
\end{equation}
\begin{equation}
\label{e:concdim}
D = \inf\{\dim_\cH(S) :  S\subset \O \text{ compact, and } \cE(S)>0\},
\end{equation}
with the convention that $D=n$ if the collection over which the infimum is taken is empty. Roughly speaking, lower values of $d(x,\cE)$ and $D$ correspond to more severe energy concentration and thus more singular solutions $u$.  The local dimension is a standard geometric measure theoretic quantity, see \cite{Mattila}, while the concentration dimension was first introduced in \cite{ShvydConc}, together with the energy measure itself.  Originally, the energy measure was developed in conjunction with a study of energy concentration and drain phenomena, especially for the purpose of excluding certain cases of self-similar blowup.  

\subsection {Overview of Main Results}
The present work breaks into three main pieces.  In the first part (Section \ref{s:dE}), we give a systematic study of the energy measure.  In particular, we discuss a connection between the so-called Onsager singular set, the energy measure, and the local energy equality \eqref{e:localee}. Furthermore, we relate the concentration dimension of $\cE$ to the phenomenon of concentration of energy, and we use basic tools of measure theory to understand the defect measure $\th = \cE - |u(0)|^2\dx\mres \O$.

In the second part (Sections \ref{s:locdim}--\ref{s:NSE}), we prove local energy bounds on $u$.  Under the assumption that $u\in L^{q,*}(-1,0;L^p)$ (and additional assumptions if $q=\infty$), our main results are stated in terms of bounds on the quantity 
\begin{equation}
A(r,x_0) = \frac{1}{r^\b}\sup_{-r^{\a}< t < 0} \int_{B_r(x_0)} |u(x,t)|^2 \dx,
\end{equation}
where 
\[
\a = \frac{q}{q-1}\left(1 + \frac np\right);
\quad \quad 
\b = \frac{q}{q-1}\left( n - \frac{2n}{p} - \frac{2+n}{q}\right).
\]
The definitions of $\a$ and $\b$ are motivated by considerations of scale-invariance; see Section \ref{s:locdim} below. (However, we note here that $p$ and $q$ must be such that $p\ge 3$ and $\b\ge 0$.) We will prove that on any compact set $K\ss \ss\O$, there exists $R>0$ and a constant $C$ such that for any $r\in (0,R)$ and any $x_0\in K$, we have $\sup\{A(r,x_0):x_0\in K,\,0<r<R\}\le C$.  If $q=\infty$, then the extra required hypothesis is either that $u\in L^\infty L^p$ (i.e. strong in time), or that $u$ satisfies the explicit power-law bound $\|u(t)\|_{L^\infty(\R^n)} \le C|t|^{-1/q}$.  In the strong-in-time case, we will have the same conclusion as before; in the power-law bound case, we will prove that $\sup\{A(r,x_0):x_0\in\R^n,\, r>0\}\le C$.  For the detailed statement of these bounds, see Section \ref{s:locdim}. Finally, we note that we can obtain similar bounds on $A(r,x_0)$ even if $p<3$ in some cases, if $u$ is a solution of the Navier-Stokes equations; see Section \ref{s:NSE}.

The uniform bounds on $A(r,x_0)$ just mentioned have several important consequences.  First and foremost, we consider the special case $(p,q)=(\infty,2)$ under the power-law assumption.  If $n=3$ and $u$ solves the Navier-Stokes equations, then the hypothesis is precisely the Type-I condition \eqref{e:typeINSE}.  In this case, the bound $A(r,x_0)\le C$ actually implies that $u$ satisfies a certain Type-I \textit{in space} condition, which is enough to guarantee energy equality.  For details of this argument, see Section \ref{s:NSE}.  For now, we record the end result as our main Theorem:
\begin{THEOREM}
\label{t:typeI}
Let $(u,p)$ be a solution to the Navier-Stokes initial value problem \eqref{e:moment}--\eqref{e:IC} which satisfies \eqref{e:pdef} and is regular on the time interval $[t_0, t_1)=[-1,0)$.  If $u$ experiences Type-I in time blowup \eqref{e:typeINSE} at $t=0$, then $u$ still preserves the energy law on the closed interval $[-1,0]$ including the first blowup time.
\end{THEOREM}

The second consequence of our uniform bounds on $A(r,x_0)$ is that we obtain a uniform lower bound on the local dimension $d(x_0, \cE)$ of the energy measure for points $x_0\in \O$ (or $x_0\in \R^n$); namely $d(x_0, \cE)\ge \b$.  This follows straightforwardly from the definitions of $A(r, x_0)$ and $d(x_0, \cE)$.  Actually, we can say slightly more. We make a conclusion about not just the local dimension, but also about uniform boundedness of the upper $\b$-density of the energy measure: $\Theta^{*\b}(\cE, x) = \limsup_{r\to 0} (2r)^{-\b}\cE(B_r(x))$, see \cite{Mattila}. This quantitatively expresses the fact that $\cE$ behaves no worse than the Hausdorff $\b$-dimensional measure under a given $L^q L^p$ condition on $u$.  See Section \ref{s:ud}.  

By a covering argument, the bounds on $A(r,x_0)$ give the same lower bound for the concentration dimension as for the lower local dimension: $D\ge \b$.  For the details of this covering argument, see Section \ref{s:ud}.  However, if $u\in L^q L^p$ for some $p$ and $q$ such that $p<\infty$ and $\b>0$, this bound is demonstrably not optimal; in this case we give more refined bounds for $D$ using different methods described below.

And finally, we use techniques developed in \cite{LS2016b} to give lower bounds for $D$ which are (in most cases) strictly better than the bounds mentioned above.  For the Euler equation, we have the following improvement:
\begin{equation}\label{e:introD}
u \in L^q L^p(\O)  \Rightarrow  D \geq n - \frac{\frac2q}{1-\frac{2}{p}-\frac1q}.
\end{equation}
The latter is strictly larger than $\b$ if $p<\infty$ and $\b>0$. Consequently, we find that if the set of singular points at time $t=0$ has dimension lower than stated in \eqref{e:introD}, then the energy of the solution is conserved; see \thm{t:oldmethod} for the full statement. For the Navier-Stokes equations, the improvement is even more dramatic in view of the Caffarelli-Kohn-Nirenberg Theorem, \cite{CKN}, which tells us that the Hausdorff dimension $d$ of the singularity set is at most $1$ (see Section \ref{s:suitable} for more details). Under a range of $L^q L^p$ conditions, this automatically implies energy equality as a consequence of our results in \cite{LS2016b}. \thm{t:oldNSE} states the full range of bounds and energy law criteria in this case.

\subsection{Additional Remarks}

In the power-law assumption case, our bounds on $A(r,x_0)$ constitute an infinitesimal improvement over a result of \cite{ShvydConc}. In that paper, almost the same uniform bound $A(r,x_0)\le C$ was proved, except that $\a$ and $\b$ are replaced by $\a+\d$ and $\b-\d$ in that setting.  In particular, the lower bounds we obtain on the local dimension are already known from \cite{ShvydConc}, since the local dimension is insensitive to the presence of the $\d$'s.  On the other hand, removing the $\d$'s is crucial in order to prove Theorem \ref{t:typeI}, which is only available with the sharper estimate.  Key in obtaining the improved bound is a modified inequality for the pressure, which depends on $u$ in a way that is essentially \textit{local} in nature.  

For all cases other than the power-law assumption, the bounds on $A(r,x_0)$ that we establish are, to the best of our knowledge, completely new.  We use an iteration procedure reminiscent of the partial regularity theory for the Navier-Stokes equations, c.f. \cite{Scheffer}, \cite{CKN}.  Especially in the critical case $p=\infty$, our choice of the scaling $\a$ plays an important role in preserving smallness from step to step.  This scaling is different, however, from the usual Navier-Stokes scaling (where $\a=2$), except on the Prodi-Serrin line $\frac3p + \frac2q = 1$.  Above this line (i.e. when $\frac3p + \frac2q > 1$), the dissipation is of lower order, according to our scaling.  This partially explains why (when $p\ge 3$) our method gives the same bounds for the Euler and Navier-Stokes case, rather than an improved statement for Navier-Stokes due to the dissipation.

Before beginning in earnest, we make one more remark in order to bring attention to a recent work of Chae and Wolf \cite{Chae}, which we learned of during the review period of the present paper.  In that paper, the authors consider Type-I blowup for the Euler equations, and it is proved that under the assumption
\[
\sup_{-1<t<0} (-t)\|\n u(t)\|_{L^\infty}< \infty,
\]
the energy measure has no atoms.  Actually, their statement is more general than this, but we cite it in simplified form because their definition of the energy measure is slightly different from ours.

\section{The Energy Measure}\label{s:dE}

\subsection{Energy Measure and the Local Energy Equality}

As mentioned in the Introduction, the following is proved in \cite{RSSingFlows}.
\begin{lemma}
	\label{l:OR}
	If $u \in \OR(U\times [-1,0])$, $U\ss \O$, then the local energy equality holds on $U$, i.e. \emph{$\dE \mres U = |u(0)|^2\dx\mres U$}.
\end{lemma}
For the remainder of the paper, we will omit the interval $[-1,0]$ from our notation of $\OR$.  

Let us look at the classical Lebesgue decomposition of the energy measure relative to $\dx\mres \O$:
\[
\dE = f \dx\mres\O + \dmu, \quad  \dx \perp \dmu.
\]
According to the discussion above, the defect measure $\dth = \dE - |u(0)|^2 \dx\mres \O$ is nonnegative. Therefore we have $f \geq |u(0)|^2$ a.e., and $\dmu \geq 0$ in general. In light of this, it is natural to attribute a possible failure of the local energy equality to two phenomena:
\begin{itemize}
	\item Concentration: $\dmu >0$;
	\item Oscillation: $f > |u(0)|^2$.
\end{itemize}
It is easy to give one sufficient condition to rule out oscillation.   Let $U$ be the largest open set in $\O$ for which $u\in \OR(U)$ (i.e., let $U$ be the union of all such sets). Define the set  of Onsager-singular points by $\Sigma_{ons} = \O \backslash U$; this set is relatively closed in $\O$. According to the previous lemma, the defect measure $\th$ is supported on $\Sigma_{ons}$.  So, if $|\Sigma_{ons}| = 0$, then the defect measure is mutually singular to $\dx$. (Here and below, we use $|A|$ to denote the Lebesgue measure of a set $A\subset \R^n$.) Thus the above Lebesgue decomposition becomes
\[
\dE = |u(0)|^2 \dx\mres \O + \dth,  \quad \dx \perp \dth,
\]
i.e. $f = |u(0)|^2$ and $\mu = \th$.  The size of the set $\Sigma_{ons}$ is related to the phenomenon of intermittency in fully developed turbulence and is out of scope of this present paper.

\subsection{Concentration Dimension} Generally, the smaller the set on which $\cE$ is concentrated, the more severe we view the blowup.  The concentration dimension assigns a numerical value to  the  concentration of the energy measure, namely the smallest Hausdorff dimension of a set of positive $\cE$-measure:
\[
D = \inf\{\dim_\cH(S) :  S\subset \O \text{ compact, and } \cE(S)>0\}.
\]
We recall that if the above family of sets is empty, then we set $D=n$ by convention. This situation occurs when the energy is drained from the domain $\O$, a scenario not excluded at the time of blowup. Generally, if $D=n$ one might say that the measure has no lower dimensional concentration. This, however, does not rule out the presence of a singular component $\dmu$. It can still be concentrated on a set of Lebesgue measure zero, but of dimension $3$. If, however, we have $D<n$, then the concentration pertains to the singular part $\dmu$ only, since obviously $f \dx$ vanishes on any subset of $\O$ with dimension less than $n$. It is in the case $D<n$ only where we can properly address the concentration issue.

By analogy with the set of Onsager-singular points, which encompasses the maximal set on which the energy equality may fail, we introduce a corresponding set of singularities which encompasses any possible concentration of the energy measure. Again, we define a set $\Sigma$ as complementary to
\begin{equation}\label{d:Sigma}
\R^n \backslash \Sigma = \{ x\in \R^n:  \exists \text{ open }  U, x\in U, \exists p>2, \exists \e>0 : u\in L^\infty(-\e,0; L^p(U)) \}.
\end{equation}
Clearly $\R^n \backslash \Sigma $ is open, so $\Sigma$ is closed.

\begin{lemma}\label{l:dEdx}
	The energy measure \emph{$\dE$} is absolutely continuous with respect to Lebesgue measure \emph{$\dx$} on $\O \backslash \Sigma$. Hence, \emph{$\supp \dmu \ss \Sigma$}.
\end{lemma}
\begin{proof}
	Let $A \ss \O \backslash \Sigma$ be a set of Lebesgue measure zero. We need to show $\cE(A)=0$. By considering the sequence $A \cap \{x \in A: \dist{x}{\p \O \backslash \Sigma} > 1/k \}$ we may assume without loss of generality that $A$ has a positive distance to $\p \O \backslash \Sigma$. Moreover, by inner regularity we may assume that $A$ is compact. Thus, $A$ is compactly embedded into $\O \backslash \Sigma$. For every point $x\in A$ we can find an open neighborhood $U_x$, $\e_x>0$ and $p_x>2$ as in the definition \eqref{d:Sigma}. By compactness there is a finite subcover, and hence we can pick the smallest of all $\e$'s and $p$'s to find a compactly embedded open neighborhood $U$ of $A$ such that $u\in L^\infty(-\e,0; L^p(U))$. We further reduce $U$ to $V \ss U$ (still containing $A$) with $|V| < \d$. Find a function $\s \in C_0(V)$, $0\leq \s\leq 1$, and $\rest{\s}{A} =1$. Then
	\[
	\cE(A) \le \int \s \,\dE = \lim_{t\to 0} \int |u(t)|^2 \s\dx \le \|u\|_{L^\infty(-\e,0; L^p(U))}^2 |V|^{\frac{p-2}{p}} < C \d^{\frac{p-2}{p}}.
	\]
	This shows that $\cE(A) = 0$, and the lemma follows.
\end{proof}
Let us note that since in general there is no relationship between the sets $\Sigma$ and $\Sigma_{ons}$, we cannot claim that the local energy equality necessarily holds on the set $\O \backslash \Sigma$. Instead, we only rule out the concentration phenomenon, while oscillation may still occur. To summarize, the lemma claims
\[
\dE \mres (\O \backslash \Sigma) = f \dx \mres (\O \backslash \Sigma), \quad f \geq |u(0)|^2.
\]
Lemma \ref{l:dEdx} has two additional immediate consequences.  By definition of $\Sigma$, we have $\Sigma=\emptyset$ if $u\in L^\infty(-1,0;L^p(\O))$ for some $p>2$, in which case $\mu$ is trivial (since $\supp \mu \subset \Sigma$ by Lemma \ref{l:dEdx}).  This rules out any concentration and allows us to conclude that $D=n$ in this case.  The second consequence is that we may take the infimum in the definition of $D$ over sets that are contained in $\Sigma$, rather than over general compact subsets of $\O$.  We record these two corollaries for reference:

\begin{corollary}\label{c:Lp} If $u\in L^\infty(-1,0; L^p(\O))$ for some $p>2$, then the energy measure suffers no concentration. That is, $\Sigma = \emptyset$, and therefore $D=n$.
\end{corollary}

\begin{corollary}\label{c:newD} The dimension of concentration is equal to
\emph{
\[
D = \inf\{\dim_\cH(S) :  S\subset \O\cap \Sigma \text{ compact, and } \cE(S)>0\}.
\]}
\end{corollary}
\begin{proof} Let us denote the new dimension $D'$ for reference.  Clearly, $D'\geq D$, since the new infimum is taken over a smaller family.  Let us address the case $D=n$ separately. In this case $D'=n$, either by convention (if no sets $S$ are available), or because $D'\geq D$. If $D<n$, we can pick $\e>0$ and a set $S$ with $\dim_{\cH}(S) \leq D+\e<n$ such that $\cE(S)>0$. However, $|S| = 0$, and hence by \lem{l:dEdx} we have $\cE(S\backslash \Sigma) = 0$. We can then replace $S$ with $S\cap \Sigma$ without changing its $\cE$-measure.  But then $\dim_\cH(S\cap \Sigma) \leq \dim_\cH(S)$, while $\cE(S\cap \Sigma)>0$; thus $D'\le \dim_\cH(S)<D+\e$. This proves the statement.
\end{proof}

\subsection{Navier-Stokes and Suitable Weak Solutions}
\label{s:suitable}

In the case of the NSE, the partial regularity theory of Caffarelli, Kohn, and Nirenberg \cite{CKN} allows us to restrict attention to lower-dimensional singular sets at time $t=0$, even though we have not assumed that our solution $u$ is suitable.  Indeed, assume $(u,p)$ satisfies \eqref{e:moment}--\eqref{e:pdef}, and assume $u$ be a Leray-Hopf weak solution on $[-1,0]$ (which is regular on $[-1,0)$). 	Let $(\widetilde{u}, \widetilde{p})$ be a suitable weak solution on $[-\frac12, \infty)$, with initial data $\widetilde{u}(-\frac12) := u(-\frac12)$ and pressure $\widetilde{p}=R_i R_j(u_i u_j)$.  Assume without loss of generality that $\widetilde{u}$ is weakly continuous in time; this can be achieved by modifying $\widetilde{u}$ on a Lebesgue null set of times.  Then By weak-strong uniqueness, we have $(u,p)=(\widetilde{u}, \widetilde{p})$ on $[-\frac12, 0)$.  Then, by weak continuity in time, we have $u(0)=\widetilde{u}(0)$.  Since $\widetilde{u}$ is suitable, the Caffarelli-Kohn-Nirenberg Theorem implies that the parabolic 1-dimensional Hausdorff measure of $\widetilde{S}$ is  $0$, where $\widetilde{S}$ is the set of singular points of $\widetilde{u}$. Note that by the Prodi-Serrin criterion, $\widetilde{u}$ is $C^\infty$ in the spatial variables on the complement of $\widetilde{S}$. This immediately implies that $\Sigma_{ons}\cup \Sigma \ss \widetilde{S}\cap\{t=0\}$, and hence that the dimensions of $\Sigma_{ons}$ and $\Sigma$ are both at most $1$.  Combining the fact that $\dim_\cH(\Sigma)\le 1$ with Corollary \ref{c:newD}, we may also conclude that $D\ge 1$. These facts will be used in Section \ref{s:NSEConcdim}.

\subsection{Upper Densities, Local Dimension, and Concentration Dimension}
\label{s:ud}

As mentioned in the Introduction, the uniform bounds $A(r,x_0)\le C$, or 
\begin{equation}
\label{e:Abd}
\sup_{-r^{\a}< t < 0} \int_{B_r(x_0)} |u(x,t)|^2 \dx \le Cr^\b, 
\end{equation}
imply slightly more than just a lower local dimension of at least $\b$.  Once we know that $\b$ is a lower bound, we can refine our geometric measure-theoretic statement by asserting the finiteness of the upper $\b$-density of $\cE$.  Let us recall that for $0\le s<\infty$ and $\mu$ a Radon measure, the upper $s$-density of $\mu$ at $x\in \R^n$ is given by
\begin{equation}
\label{e:sdensity}
\Theta^{*s}(\mu, x) = \limsup_{r\to 0} (2r)^{-s}\mu(B_r(x)).
\end{equation}
If $\mu$ has finite $s$-density at $x$, then, roughly speaking, $\mu$ behaves near $x$ like $s$-dimensional Hausdorff measure on an $s$-dimensional set:  $\mu(B_r(x)) \lesssim r^s$.

%It is somewhat awkward to constantly discuss uniform boundedness of various $s$-densities of the energy measure in our present context; we essentially wish to vary $s$ to obtain the optimal value for which the $s$-density is uniformly bounded.  As such, the lower local dimension is `almost' the right quantity to work with, as finiteness of the $s$-density of a measure $\mu$ at $x$ easily implies that the lower local dimension of $\mu$ at $x$ is at least $s$.  However, working with the lower local dimension does cause a small loss of information.  Therefore, we will henceforth also state our results regarding local dimension in terms of the upper $\b$-density, since Proposition \ref{p:A} tells us that this is the correct $s$-density to consider.

Let us also give the covering argument alluded to in the Introduction, which relates bounds of the form \eqref{e:Abd} to lower bounds on $D$.  Suppose we have $\mu(B_r(x)) \leq C(K) r^s$, for all $x\in K$ and all sufficiently small $r>0$. Then for any set $S\ss \O$ with $\dim_{\cH}(S) <s$ and for any compact subset $K\ss S$, we have $\mu(K) \leq \sum_i \mu(B_{r_i}(x_i)) \leq C \sum r_i^s  \to 0$, as the cover closes on $K$. So, $\mu(K) = 0$, and hence $\mu(S) = 0$ by inner regularity.  This shows that $D\geq s$.
	
Interestingly, the above argument does not proved a sharp bound on $D$ from below, due to the fact that the covering argument using additivity of $\cE$ is simply not optimal. One obtains a better estimate by examining the cover in its entirety via the local energy inequality, c.f. Section \ref{s:ConcDim}.

\section{Local Dimension of the Energy measure}
\label{s:locdim}
Let $u$ be a classical solution to the Euler equation on time interval $[-1,0)$. Let $\O \ss \R^n$ be an open bounded domain. Suppose that  $u \in L^{q,*}(-1,0; L^p(\O))$ for some $p\geq 3$ and $q > 1$. Out of the classical two parameter family of scaling symmetries of the Euler equation there is one that leaves the $L^qL^p$-condition invariant, namely
\begin{equation}\label{e:scale}
u(x,t)\mapsto \l^{\a-1} u(\l x, \l^\a t), \text{ where } \a = \frac{q}{q-1}(1 + \frac np).
\end{equation}
With the scaling \eqref{e:scale} in mind, we state our main results for this section as follows. 
\begin{proposition}\label{p:A}
Let $\O \ss \R^n$ be an open domain and $K \ss \O$ a compact subset. Suppose $u$ is a solution to the Euler equations satisfying $u \in L^{q,*}(-1,0;L^p(\O))$ with $3\le p<\infty$, or satisfying $u \in L^{q}(-1,0;L^\infty(\O))$, and in both cases
\begin{equation}\label{e:pq}
\frac{2n}{p} + \frac{2+n}{q} \le n.
\end{equation}
Then there exist positive constants $R = R(n,p,q,u,K)$ and $C_0 = C_0(n,p,q,u,K)$ such that for all $r\in (0,R)$ we have
\emph{
\begin{equation}
\label{e:Apq}
\sup_{-r^{\a} < t < 0, \,x_0 \in K} \int_{B_r(x_0)} |u(x,t)|^2\dx \le C_0\,r^{\b},
\end{equation}}
where $\a = \frac{q}{q-1}\left(1 + \frac np\right)$ and $\b = \frac{q}{q-1}\left( n - \frac{2n}{p} - \frac{2+n}{q}\right)$.
\end{proposition}
Note that \eqref{e:pq} is precisely equivalent to the condition $\b\ge 0$.  Our other main result of this section is the following:
\begin{proposition}
\label{p:typeI}
Suppose $u$ is a solution to the Euler equation which is regular on $[-1,0)$ and satisfies the bound $\|u(t, \cdot)\|_{L^\infty} \le c_0 |t|^{-1/q}$, $\frac{n+2}{n} \le q$.  Then there exists a constant $C=C(u,n,q)$ such that
\emph{
\begin{equation}\label{e:typeI}
\sup_{-1<t<0,\,x_0\in \R^n}\int_{|x-x_0|<r} |u(x,t)|^2\dx \le C r^{n-\frac{2}{q-1}}.
\end{equation}}
\end{proposition}

We define several scale-invariant quantities relating to the scaling \eqref{e:scale}, which will be used in the proof of Proposition \ref{p:A}.  First, denote $Q_r:=B_r\times (-r^\a,\,0)$, and let $(p)_r = \frac{1}{|B_r|} \int_{B_r} p(x) \dx$ denote the average of $p$ on $B_r$. We define
\begin{align*}
A(r) &= \frac{1}{r^\b}\sup_{-r^{\a}< t < 0} \int_{B_r} |u(x,t)|^2 \dx, & \text{ (energy)}\\
G(r) & = \frac{1}{r^{\b+1}}\int_{Q_r} |u(x,t)|^3\dx\dt, & \text{ (flux)}\\
P(r) & = \frac{1}{r^{\b+1}}\int_{Q_r} |p-(p)_r| |u| \dx\dt,& \text{ (pressure)}.
\end{align*}

\begin{REMARK}
\label{r:Morrey}
The inequality \eqref{e:typeI} can be expressed as $ u\in L^\infty \cM^{2,n-\frac{2}{q-1}}$, where $\cM^{p,\l}$ is the Morrey space with integrability $p$ and rate index $\l$.  This observation plays an important role in the proof of Theorem \ref{t:typeI}; see the end of Section \ref{s:NSE}.
\end{REMARK}

We devote the next three subsections to the proof of Proposition \ref{p:A}; Proposition \ref{p:typeI} is proved in Section \ref{s:typeI}.

\subsection{Essential estimates} The proof of \prop{p:A} is executed by induction on scales according to the sequence of bounds $A(r) \to G(r) \to P(r) \to A(r/2)$. Although the details of the iteration procedure depend on which hypothesis is used, the proofs of both cases rely on common estimates on the quantities $A,G,P$. We start with an elementary $L^3 L^3$ estimate away from the boundary.

\begin{claim}\label{c:L3} Suppose $u \in L^{q,*}(-1,0; L^p(\O))$, where $(p,q)$ satisfies $3\le p\le \infty$ and \eqref{e:pq}.  Then $u\in L^3 L^3(\O)$ and $p\in L^{3/2} L^{3/2}(\O_\e)$ for any $\O_\e = \{ x \in \O: \dist{x}{\p \O} > \e\}$.
\end{claim}
\begin{proof} Note that $q > p/(p-2)$, by \eqref{e:pq}. By reducing $q$ we can assume without loss of generality that $u\in L^qL^p$, i.e. strong in $q$, yet $q > p/(p-2)$ still holds. Then finiteness of $\|u\|_{L^3 L^3(\O)}$ follows easily by interpolation and H\"older's inequality:
\[
\int_0^T \|u\|_{L^3(\O)}^3\dt
\le \int_0^T \|u\|_{L^2(\O)}^{\frac{2(p-3)}{p-2}} \|u\|_{L^p(\O)}^{\frac{p}{p-2}}\dt \le \|u\|_{L^\infty L^2(\O)}^{\frac{2(p-3)}{p-2}} \|u\|_{L^q L^p(\O)}^{\frac{p}{p-2}} T^{1 - \frac{p}{q(p-2)}}.
\]

Let $\eta:\R^n\to \R$ be a smooth function such that $\eta\equiv 1$ on $\O_{\e/2}$ and $\supp \eta \subset \O$.  Let $R_i, R_j$ denote the Riesz transforms on $\R^n$, and let
\[
K_{ij}(y) = \frac{n y_i y_j - \d_{ij}|y|^2}{n\w_n|y|^{n+2}}
\]
denote the kernel of $R_i R_j$. (Here $\d_{ij}$ is the Kronecker delta, and $\w_n$ is the volume of the unit ball in $\R^n$.)  Since $p=R_i R_j(u_i u_j)$, we can use the boundedness of the Riesz transforms on $L^{3/2}$ to estimate $\|p\|_{L^{3/2}(\O_\e)}$ as follows:
\begin{align*}
\|p\|_{L^{3/2}(\O_\e)}
& \le \|R_i R_j(\eta u_i u_j)\|_{L^{3/2}(\O_\e)} + \|R_i R_j((1-\eta)u_i u_j) \|_{L^{3/2}(\O_\e)} \\
& \le C\|\eta u_i u_j\|_{L^{3/2}(\R^n)} + \left\| \int_{\R^n} K_{ij}(\cdot-y)(1 - \eta(y)) u_i(y) u_j(y) \dy \right\|_{L^{3/2}(\O_\e)} \\
& \le C\|u\|_{L^3(\O)}^2 + C\left\| \int_{\O_{\e/2}^c} \frac{|u(y)|^2}{|\cdot-y|^n}\dy\right\|_{L^{3/2}(\O_\e)} \\
& \le C\|u\|_{L^3(\O)}^2 + C\e^{-n} \|u\|_{L^2(\O)}^2 |\O_\e|^{2/3}.
\end{align*}
From here it is obvious that taking the $L^{3/2}$ norm in time yields a finite quantity.
\end{proof}

In what follows we assume without loss of generality that $0 \in \O$ and $r < \frac12 \dist{0}{\p \O}$. We start with local energy equality
\begin{equation}
\label{localee}
\int |u(t)|^2 \s(t)\dx = \int |u(s)|^2\s(s)\dx + \int_s^t \int |u|^2 \p_t\s\dx\dtau + \int_s^t \int (|u|^2 + 2p)u\cdot \n \s\dx\dtau,
\end{equation}
valid for any $\s\in C_0^\infty([-1,0]\times \O)$ and $-1\le s\le t<0$. Let $\psi:[0,\infty)\to \R$ be a smooth nonincreasing function such that $\psi(z) = 1$ for $z\le 1$, and $\psi(z) = 0$ for $z\ge 2$.  Define $\phi_r(x,t) = \psi(|x|/r)\psi(|t|/r^\a)$, so that $\phi_r$ is $1$ on $Q_r$ and zero outside $Q_{2r}$.  Putting $\s = \phi_r$ in the local energy equality yields
\begin{equation}
\label{localei}
\sup_{-r^\a\le t\le 0} \int_{B_r} |u(t)|^2\dx
\le \int |u|^2 |\p_t\phi_r|\dx\dtau + \int |u|^3|\n\phi_r|\dx\dtau + 2\int |p-(p)_{r}||u||\n \phi_r| \dx\dtau,
\end{equation}
Note that
\[
|\phi_r|\le \chi_{Q_{2r}},
\quad
|\n \phi_r|\le C r^{-1}\chi_{Q_{2r}},
\quad
|\p_t \phi_r|\le C r^{-\a}\chi_{Q_{2r}}.
\]
Evaluating the above at half the radius $r \to r/2$ and dividing through by $r^{\b}$ yields \footnote{In all intermediate estimates we omit constants $C$ which are independent of the radius.}
\[
\begin{split}
A(r/2) & \le
\frac{1}{r^{\a+\b}} \int_{Q_{r}} |u|^2\dx\dtau  + \frac{1}{r^{1+\b}}\int_{Q_{r}} |u|^3\dx\dtau + \frac{1}{r^{1+\b}}\int_{Q_{r}} |p-(p)_{r}||u| \dx\dtau\\
& \le \frac{r^{2(1+\b)/3}}{r^{\a+\b}}\cdot  (r^{n+\a})^{1/3} \left(\frac{1}{r^{1+\b}} \int_{Q_{r}} |u|^3\dx\dtau\right)^{2/3} + G(r) +  P(r) \\
& \le G(r)^{2/3} + G(r) + P(r) .
\end{split}
\]
We have obtained
\begin{equation}\label{e:AG}
A(r/2)\leq C[ G(r)^{2/3} + G(r) + P(r)].
\end{equation}
Next, we establish a bound on the flux $G(r)$ in terms of $A(r)$.
\begin{align*}
G(r) & = r^{-\b - 1} \int_{-r^\a}^0 \int_{B_{r}} |u(x,t)|^3\dx\dt \\
& \le r^{-\b-1} \int_{-r^\a}^0 \left( \int_{B_{r}} |u(x,t)|^2\dx \right)^{\frac{p-3}{p-2}} \left( \int_{B_{r}} |u(x,t)|^p \dx \right)^{\frac{1}{p-2}}\dt \\
&\le r^{-\frac{\b}{p-2}-1} \int_{-r^\a}^0 \left( \frac{1}{r^{\b}}\int_{B_{r}} |u(x,t)|^2\dx \right)^{\frac{p-3}{p-2}} \left( \int_{B_{r}} |u(x,t)|^p \dx \right)^{\frac{1}{p-2}}\dt \\
&\le r^{-\frac{\b}{p-2}-1} A(r)^{\frac{p-3}{p-2}} \int_{-r^\a}^0 \left( \int_{B_{r}} |u(x,t)|^p \dx \right)^{\frac{1}{p-2}}\dt.
\end{align*}
Denote $f(t) = \| u\|_{L^p}$. Under the time integral we have a quantity bounded by $f^{\frac{p}{p-2}}$. We know, however, that $f\in L^{q,*}$, and that under our assumption \eqref{e:pq} we have $\frac{p}{p-2} < q$. This allows us to extract the same asymptotic behavior in $r$ as if $f$ were in the strong $L^q$-space. Indeed,
\[
\int_{-r^\a}^0 f(t)^\frac{p}{p-2} \dt \le  \frac{p}{p-2} \int_0^\infty \l^{\frac{2}{p-2}} \min \{ | \{ f > \l \} | , r^\a \} \dl.
\]
Using that $| \{ f > \l \} | \leq C / \l^q$ and splitting the integral, we obtain a bound by $r^{\a - \frac{\a p }{q(p-2)}}$. Adding this power of $r$ to the already present power $-\frac{\b}{p-2} - 1$ gives a net power $0$. Thus, we obtain
\begin{equation}\label{e:GA}
G(r) \leq C  A(r)^{\frac{p-3}{p-2}}.
\end{equation}
The case $p<\infty$ has a clear advantage of yielding a power of $A$ smaller than $1$, while the case $p=\infty$ is critical. The latter can be handled in a similar way under the strong $L^q$ in time condition: making the obvious adjustments for $p=\infty$ in the estimates on $G(r)$ above, we obtain the alternative bound 
\begin{equation}\label{e:GA2}
G(r) \leq C \e(r) A(r),
\end{equation}
where 
\begin{equation}
\label{e:epsilon}
\e(r) = \|u\|_{L^q(-r^\a,0;L^\infty)}.
\end{equation}
The small parameter $\e(r)$, which vanishes as $r\to 0$, allows us to compensate for the accrued constant $C$ and close the circle of bounds $A(r) \to G(r) \to P(r) \to A(r/2)$ by induction.  (See below for more details.) 

In order to handle the weak case of $L^{q,*}L^\infty$, we need an explicit power bound in time, and we use a more subtle argument, c.f. Section \ref{s:typeI}.

Turning now to the pressure term, we recall the following local pressure inequality.
\begin{LEMMA}
	\label{l:pressure}
	There exists an absolute constant $c$ such that whenever $p\in L^{3/2}(B_\rho)$ and $-\Delta p = \p_i \p_j(u_i u_j)$ a.e. on $B_\rho$, then for any $r\in (0,\rho/2]$ we have \emph{
		\begin{equation}
		\label{e:locpress}
		\|p - (p)_r\|_{L^{3/2}(B_r)} \le c \|u\|_{L^3(B_{2r})}^2 + cr^{\frac23 n + 1} \int_{2r<|y|<\rho} \frac{|u|^2}{|y|^{n+1}}\dy + c \frac{r^{\frac23 n + 1}}{\rho^{\frac23 n + 1}}\left( \int_{B_\rho} |u|^3 + |p|^{3/2} \dy\right)^{\frac23}.
		\end{equation}}
\end{LEMMA}
This inequality is proven for $n=3$ in in Lemma 15.12 in \cite{RRS}. (Actually, the inequality is stated there with a time integral; \eqref{e:locpress} is obtained as an intermediate step in their proof.) The $n$-dimensional case is adaptable by simply replacing $3$ with $n$ in the appropriate places.  Since we are considering only times prior to the first possible blowup, the hypotheses are valid for the pair $(u(t),p(t))$ in either the Euler or Navier-Stokes case.  Therefore we may use this estimate for either set of equations.  Note that subtracting off the average on $B_{r_k}$ is crucial in order to obtain this local estimate, since $p$ depends nonlocally on $u$. 

Choose $R < \frac12 \dist{0}{\p\O}$ and write $r_j = R / 2^j$ for all $j\in \N$.  Then using the local pressure inequality with $r=r_k$ and $\rho = R/2 = r_1$, we will obtain an estimate on $P(r_k)$. First, we split the integral in the second term into dyadic shells $r_{j+1}\le |y|<r_j$ and estimate $|y|^{-n-1}$ pointwise on each of these shells before replacing the shells with balls.  The result is
\[
\|p - (p)_{r_k}\|_{L^{3/2}(B_{r_k})} \le \|u\|_{L^3(B_{r_{k-1}})}^2 +r_k^{\frac23 n + 1}\sum_{j=1}^{k-2} r_j^{-(n+1)}\|u\|_{L^2(B_{r_{j}})}^2 + \frac{1}{2^{(\frac23 n + 1)k}} g(t),
\]
where $g(t)$ is some function belonging to $L^{3/2}_t$, by Claim~\ref{c:L3}.  We turn this into a bound on $P(r_k)$ as follows:
\[
\begin{split}
P(r_k)
& \le \frac{1}{r_k^{1+\b}}\int_{-r_k^\a}^0 \|u\|_{L^3(B_{r_{k-1}})}^2 \|u\|_{L^3(B_{r_k})}\dt + r_k^{\frac23 n -\b} \sum_{j=1}^{k-2}r_j^{-n-1} \int_{-r_k^\a}^0 \|u\|_{L^2(B_{r_{j}})}^2 \|u\|_{L^3(B_{r_k})}\dt \\
& \quad + \frac{1}{2^{(\frac23 n + 1)k}\ r_k^{1+\b}}\int_{-r_k^\a}^0 g(t) \|u\|_{L^3(B_{r_k})} \dt\\
& \le G(r_{k-1}) +  r_k^{\frac23 n -\b} \sum_{j=1}^{k-2} r_j^{-n-1}  r_{j}^{\frac n3} \int_{-r_k^\a}^0 \|u\|_{L^3(B_{r_k})} \|u\|_{L^3(B_{r_j})}^2 \dt + \frac{1}{2^{(\frac23 n + 1)k}\ r_k^{2(1+\b)/3}}G(r_k)^{1/3} \\
& \le G(r_{k-1})  + r_k^{\frac23 n -\b} \sum_{j=1}^{k-2} r_j^{-n-1}  r_{j}^{\frac n3} r_j^{\b +1} G(r_k)^{1/3} G(r_j)^{2/3} + \frac{1}{2^{(\frac23 n + \frac13 - \frac23 \b)k}\ R^{2(1+\b)/3}} G(r_k)^{1/3}
%\\
%\intertext{(using that the powers of $r_j$ in the sum add up to $\b - \frac23 n$, and $\frac23 n + \frac13 - \frac23 \b >0$)}
%& \leq  \max\{ G(r_1),\ldots, G(r_k) \} +  R^{- 2(1+\b)/3} G(r_k)^{1/3}.
\end{split}
\]
Using the fact that the powers of $r_j$ in the sum add up to $\b$, and the fact that $\frac23 n + \frac13 - \frac23 \b>0$,  We have obtained the following:
\begin{equation}\label{e:PG}
P(r_k) \leq  C \max\{ G(r_1),\ldots, G(r_k) \} + C R^{- 2(1+\b)/3} G(r_k)^{1/3},
\end{equation}
with $C$ independent of $k$ and $R$ in the range $R < \frac12 \dist{0}{\p\O}$.

\subsection{Case $u\in L^{q,*}L^p$, $3\leq p<\infty$}
Let us fix an arbitrary initial radius $R < \frac12 \dist{0}{\p\O}$, and set a constant $A>1$ to be determined later but so that
\[
A(R) < A.
\]
This sets the initial step in the induction on $k=0,1,\ldots$. Suppose we have
\[
A(r_j) <A,
\]
for all $j \leq k$. By \eqref{e:GA} we have
\[
G(r_j) \leq C_1 A^{1 -\d},\quad \d = \frac{1}{p-2},
\]
for all $j \leq k$. In view of \eqref{e:PG},
\[
P(r_k) \leq C_2  A^{1 -\d},
\]
where $C_2$ depends on the (fixed) constant $R$.  (Note that we used that $A>1$ to bound of $A^{(1-\d)/3}$ by $A^{1-\d}$.) Returning to \eqref{e:AG}, we obtain
\[
A(r_{k+1}) \leq C_3 A^{1-\d},
\]
where still $C_3$ depends only on $R$. By setting $A > \max\{1, C_3^{1/\d}\}$ initially, we have achieved the bound
\[
A(r_{k+1}) < A,
\]
which finishes the induction.

\subsection{Case $u\in L^{q}L^\infty$}  Let us fix $R < \frac12 \dist{0}{\p\O}$, $R<1$, so that $\e(R)<1$. (Recall that $\e(r)$ is defined by \eqref{e:epsilon}.) Let $E$ denote the total energy $\|u\|_{L^2}^2$, which is independent of time on the interval $[-1,0)$. We aim to show that the bound
\begin{equation}
\label{e:qinftygoal}
A(r) < R^{-\b} E + R^{-1-\b} := A
\end{equation}
propagates through scales for initial $R$ sufficiently small. Clearly it holds for $r_0 = R$. Suppose we have
\[
A(r_j) < A
\]
for all $j \le k$. Denote $\e = \e(R)$ for convenience. Since $\e(r) \leq \e$ for $r\le R$, the bound \eqref{e:GA2} gives us
\[
G(r_j) < C_1 \e A
\]
for all $j \le k$ as well. The pressure bound \eqref{e:PG} yields
\[
P(r_k) < C_2 \e A + C_3 \e^{1/3} A^{1/3} R^{- 2(1+\b)/3}.
\]
However, $R^{- 2(1+\b)/3} < A^{2/3}$, by \eqref{e:qinftygoal}. So,
\[
P(r_k)< C_4 \e^{1/3} A.
\]
Returning to \eqref{e:AG} again, we find that
\[
A(r_{k+1}) < C_5 ( \e^{2/3} A^{2/3} + \e A + \e^{1/3} A) \leq C_6 \e^{1/3} A,
\]
where $C_6$ is independent of $R$. Picking $R$ so that $C_6 \e^{1/3}(R) <1$ finishes the induction.

\subsection{Case $\|u(t)\|_{L^\infty} \le c_0 |t|^{-1/q}$.}
\label{s:typeI}
Assume $f(t):=\|u(t,\cdot)\|_{L^\infty} \le \frac{c_0}{|t|^{1/q}}$, where $\frac{n+2}{n} \le q$. In this case we disregard the subdomain $\O$ and work on the full space only.  Let $\psi:[0,\infty)\to \R$ be a standard bump function---equal to 1 on $\{|x|\le 1/2\}$ and supported inside $\{|x|<1\}$.  Denote $\phi_r(x) = \psi(|x|/r)$ and define
\[
E(t,r)=\int |u(x,t)|^2 \phi_r(x)\dx,
\quad
\quad
E_k(t,r) = \frac{E(t,2^k r)}{2^{kn}}.
\]
Note that by definition of $E(t,r)$ we have 
\[
\|u(t)\|_{L^2(B_r)}^2 \le E(t,2r) \le \|u(t)\|_{L^2(B_{2r})}^2.
\]
We have the following Lemma:
\begin{LEMMA}
	\label{l:step1}
	There exists a constant $C_0 = C_0(u,n,q)$ such that for any $s<t<0$ and $r>0$, we have
	\emph{\begin{equation}
		\label{e:step1}
		E(t,r)\le r^n f(s)^2 + \frac{C_0}{r}\int_s^t f(\tau) \sum_{j\in \N} 2^{-j} E_j(\tau, r) \dtau.
		\end{equation}}
\end{LEMMA}

Before giving the proof, let us make a few remarks.  Fix $r>0$ and then set
\begin{equation}
t_0 = -r^{q'}, \quad  C_1= C_0 c_0 q',
\end{equation}
where $q'$ is the H\"older conjugate of $q$.  Iteration of the Lemma will eventually allow us to prove the bound
\begin{equation}
\label{e:iterM}
E(t,r) \le r^n f(t_0)^2 e^{C_1} + \frac{C_1^M}{M!} \|u\|_{L^\infty L^2}^2,
\end{equation}
valid for $t\in (t_0, 0)$ and all $M\in \N$.  The constant $C_1$ is chosen so that
\begin{equation}
\label{e:t0choice}
\frac{C_0}{r}\int_{t_0}^0 f(\tau)\dtau \le C_0 c_0 q' t_0^{1/q'} r^{-1} = C_1.
\end{equation}
Taking $M\to \infty$ in \eqref{e:iterM}, we obtain
\begin{equation}
\label{e:itMinfty}
E(t,r) \le C r^n f(t_0)^2 \le C r^{n-\frac{2}{q-1}},\quad t\in (-r^{q'}, 0),
\end{equation}
where the second inequality follows from the assumed bound on $f$ and the definition of $t_0$.  This is almost the desired bound in Proposition \ref{p:typeI}, except that the time interval does not extend to $-1$ in the negative direction.  However, the bound $E(t,r)\le Cr^{n - \frac{2}{q-1}}$ follows automatically from the assumption $\|u(t)\|_{L^\infty}\le C|t|^{-1/q}$ when $t\in [-1,r^{q'}]$.  Therefore in order to prove Proposition \ref{p:typeI}, it suffices to prove the Lemma (and the fact that the bound \eqref{e:iterM} follows).

\begin{proof}[Proof of Lemma \ref{l:step1}]
Without loss of generality we assume $x_0=0$.  We use our time-independent test function $\phi_r$ in the local energy equality \eqref{localee}, dropping the subscript for convenience:
\begin{equation}
\int |u(x,t)|^2 \phi(x)\dx = \int |u(x, s)|^2\phi(x)\dx + \int_{s}^t \int (|u|^2 + 2p)u\cdot \n \phi\dx\dtau.
\end{equation}
Applying the obvious pointwise bounds, we get
\begin{align*}
E(t,r) & \le E(s, r) + \frac{C}{r}\int_{s}^t \int_{B_{r}} |u|^3 + |p-(p)_{r}||u|\dx\dtau \\
& \le E(s, r) + \frac{C}{r}\int_{s}^t \|u\|_{L^3(B_r)}^3  + \|p-(p)_r\|_{L^{3/2}(B_r)}\|u\|_{L^3(B_r)}\dtau.
\end{align*}
	
Take $\rho\to \infty$ in the local pressure inequality \eqref{e:locpress}; the last term tends to zero because $u\in L^3(-1,0; L^3(\R^n))$.  
\begin{equation}
\|p - (p)_r\|_{L^{3/2}(B_r)}
\le c \|u\|_{L^3(B_{2r})}^2 + cr^{\frac{2n}{3} + 1} \int_{2r<|y|<\infty} \frac{|u|^2}{|y|^{n+1}}\dy.
\end{equation}
As before, we split the remaining integral into dyadic shells, estimate $|y|^{-n-1}$ on each shell, and then replace the shells with balls.  We obtain
\begin{align*}
r^{\frac{2n}{3} + 1} \int_{2r<|y|<\infty} \frac{|u|^2}{|y|^{n+1}}\dy
& \le r^{\frac{2n}{3} + 1} \sum_{j=1}^\infty (2^j r)^{-n-1} \int_{B_{2^{j+1}r}} |u|^2\dy
\le Cr^{-\frac n3} \sum_{j=3}^\infty 2^{-j} E_j(t,r).
\end{align*}
So
\begin{align*}
\|p - (p)_r\|_{L^{3/2}(B_r)}\|u\|_{L^3(B_r)}
& \le c \|u\|_{L^3(B_{2r})}^3 + Cr^{-\frac n3} \|u\|_{L^3(B_r)} \sum_{j=3}^\infty 2^{-j} E_j(t,r).
\end{align*}
Next,
\[
\|u(\tau)\|_{L^3(B_r)} \le C f(\tau)^{\frac13}E(\tau, 2r)^{\frac13}\le C f(\tau)r^{\frac n3}.
\]
Therefore
\begin{align*}
\int_{s}^t \|p-(p)_r\|_{L^{3/2}(B_{r})}\|u\|_{L^3(B_r)}\dtau
& \le C\int_{s}^t f(\tau)E(\tau,4r) + f(\tau) \sum_{j=3}^\infty 2^{-j} E_j(\tau,r)\dtau \\
& \le C\int_{s}^t f(\tau) \sum_{j=2}^\infty 2^{-j} E_j(\tau,r)\dtau,
\end{align*}
and consequently,
\begin{align*}
E(t,r) & \le E(s, r) + \frac{C}{r}\int_{s}^t \|u\|_{L^3(B_r)}^3  + \|p-(p)_r\|_{L^{3/2}(B_r)}\|u\|_{L^3(B_r)}\dtau \\
& \le r^n f(s)^2 + \frac {C_0}{r}\int_{s}^t f(\tau) \sum_{j=1}^\infty 2^{-j} E_j(\tau,r) \dtau,
\end{align*}
where $C_0$ is defined so that the last inequality holds.
\end{proof}

The final missing step in the proof of Proposition \ref{p:typeI} consists of showing that iteration of \eqref{e:step1} gives \eqref{e:iterM}. We argue as follows:

\begin{proof}[Proof of \eqref{e:iterM}]
Notice that the quantities $E_k(t,r)$ possess the following scaling property:
\[
\frac{E_j(t,2^k r)}{2^{kn}} = E_{j+k}(t,r),
\quad
j,k\in \N\cup\{0\}.
\]
We can therefore rescale the bound \eqref{e:step1} as follows:
	
\begin{align*}
E_k(t,r) & = \frac{E(t,2^k r)}{2^{kn}} \le r^n f(s)^2 + \frac{C_0}{r} \int_s^t f(\tau) \sum_{j=1}^\infty 2^{-j-k} \cdot \frac{E_j(\tau, 2^k r))}{2^{kn}}\dtau \\
& \le r^n f(s)^2 + \frac{C_0}{r} \int_s^t f(\tau) \sum_{j=k+1}^\infty 2^{-j} E_j(\tau,r)\dtau,
\end{align*}
We don't actually use the fact that this last sum starts from $j=k+1$; for our purposes it suffices to use a rougher bound, where we trivially replace the sum above with a sum over all of $\N$:
\begin{equation}
\label{e:Ekrescale}
E_k(t,r)\le r^n f(s)^2 + \frac{C_0}{r} \int_s^t f(\tau) \sum_{j\in \N}^\infty 2^{-j} E_j(\tau,r)\dtau.
\end{equation}
	
Next, we iterate to obtain \eqref{e:iterM}.  By \eqref{e:Ekrescale}, we immediately have
\begin{equation}
\label{e:Et}
E(t,r) = E_0(t,r) \le r^n f(t_0)^2 + \frac{C_0}{r}\int_{t_0}^t f(t_1) \sum_{k_1\in \N} 2^{-k_1} E_{k_1}(t_1,r) \dt_1,
\end{equation}
\begin{equation}
\label{e:Et1}
E_{k_1}(t_1,r) \le r^n  f(t_0)^2 + \frac{C_0}{r} \int_{t_0}^{t_1} f(t_2) \sum_{k_2\in \N} 2^{-k_2} E_{k_2}(t_2,r)\dt_2.
\end{equation}
Substituting \eqref{e:Et1} into \eqref{e:Et}, we get
\begin{align*}
E(t,r)
& \le r^n f(t_0)^2\left[ 1  + \frac{C_0}{r}\int_{t_0}^t f(t_1)  \dt_1\right]  + \frac{C_0}{r}\int_{t_0}^t f(t_1) \sum_{k_1\in \N} 2^{-k_1} \frac{C_0}{r} \int_{t_0}^{t_1} f(t_2) \sum_{k_2\in \N} 2^{-k_2} E_{k_2}(t_2, r)\dt_2\dt_1 \\
& \le r^n f(t_0)^2\left[ 1 + C_1 \right] + \frac{C_0^2}{r^2} \int_{t_0}^t f(t_1) \int_{t_0}^{t_1} f(t_2) \sum_{k_1,k_2\in \N} 2^{-(k_1 + k_2)} E_{k_2}(t_2, r)\dt_2\dt_1,
\end{align*}
where we have applied \eqref{e:t0choice} in order to reach the last inequality. This completes our second iteration.  We claim that at the $M$th step, we will have the bound
\begin{multline}
\label{e:stepM}
E(t,r) \le r^n f(t_0)^2 \sum_{j=0}^{M-1} \frac{C_1^j}{j!} + \frac{C_0^M}{r^M} \int_{t_0}^t f(t_1) \cdots \int_{t_0}^{t_{M-1}} f(t_M) \sum_{k_1,\ldots,k_M\in \N} \frac{E_{k_M}(t_M, r)}{2^{k_1 + \cdots + k_M}} \dt_M\cdots\dt_1.
\end{multline}
We have shown this is true for $M=1,2$.  Now we induct, using this bound to derive step $M+1$, which we will see has the same form.  Indeed, Lemma \ref{l:step1} gives us
\begin{equation}
\label{e:EtM}
E_{k_M}(t_M, r) \le r^n f(t_0)^2 + \frac{C_0}{r}\int_{t_0}^{t_M} f(t_{M+1}) \sum_{k_{M+1}\in \N} 2^{-k_{M+1}} E_{k_{M+1}}(t_{M+1},r)\dt_{M+1};
\end{equation}
substituting this into our inductive hypothesis, we get
\begin{align*}
E(t,r) & \le r^n f(t_0)^2 \left[  \sum_{j=0}^{M-1} \frac{C_1^j}{j!} + \frac{C_0^M}{r^M} \int_{t_0}^t f(t_1) \cdots \int_{t_0}^{t_{M-1}} f(t_M) \dt_M\cdots\dt_1 \right] \\
& \hspace{10 mm} + \frac{C_0^{M+1}}{r^{M+1}} \int_{t_0}^t f(t_1) \cdots \int_{t_0}^{t_M} f(t_{M+1}) \sum_{k_1,\ldots,k_{M+1}\in \N} \frac{E_{k_{M+1}}(t_{M+1},r)}{2^{k_1 + \cdots + k_{M+1}}} \dt_{M+1}\cdots\dt_1
\end{align*}
Since
\begin{equation}
\label{e:nestedint}
\frac{C_0^M}{r^M} \int_{t_0}^t f(t_1) \cdots \int_{t_0}^{t_{M-1}} f(t_M) \dt_M\cdots\dt_1
= \frac{1}{M!}\left[ \frac{C_0}{r} \int_{t_0}^t f(\tau)\dtau \right]^M
\le \frac{C_1^M}{M!},
\end{equation}
we have now proved \eqref{e:stepM}.  Having established \eqref{e:stepM}, we can prove \eqref{e:iterM} quickly. First, we clearly have
\[
\sum_{j=0}^{M-1} \frac{C_1^{j}}{j!}<e^{C_1}.
\]
To deal with the other term in \eqref{e:stepM}, we estimate each $E_{k_M}(t_M, r)$ trivially by $\|u\|_{L^\infty L^2}$ (so the entire sum can be bounded by $\|u\|_{L^\infty L^2}$).  Then we use \eqref{e:nestedint} to take care of the nested integrals.
	
Altogether we have
\begin{equation*}
E(t,r) \le r^n f(t_0)^2 e^{C_1} + \frac{C_1^M}{M!}\|u\|_{L^\infty L^2}^2,
\end{equation*}
which is \eqref{e:iterM}.
\end{proof}

\begin{corollary} Under the assumptions of either of Propositions \ref{p:A} or \ref{p:typeI}, we have the bound $d(x,\cE) \geq \b$ for all $x\in \O$.  Furthermore, the $\b$-density of $\cE$ is uniformly bounded on $\O$.  
\end{corollary}

\section{Applications to NSE}
\label{s:NSE}
If we consider the 3-dimensional Navier-Stokes equations instead of the $n$-dimensional Euler equations, we can reach conclusions in the same spirit as those above.  We describe the necessary modifications below.  Adding
\begin{equation}
\label{NSEterms}
-2\nu\int |\n u|^2 \s\dx\dtau + \nu\int |u|^2 \Delta \s\dx\dtau
\end{equation}
to the right side of \eqref{localee}, we obtain the energy equality for the NSE.  However, the first of these terms can be dropped without affecting the inequality.  The way we deal with the second term depends on the method and test function used for the Euler case; these depend in turn on the assumptions made on $u$.

We consider together the cases $u\in L^{q,*} L^p$ or $u\in L^q L^\infty$ (subject of course to the restrictions on $p$ and $q$ described above). We take $\s=\phi_r$ as constructed above when considering these cases. The second term of \eqref{NSEterms} can clearly be bounded above by
\[
\frac{C}{r^2} \int_{Q_{2r}} |u|^2 \dx\dtau.
\]
We claim that we can ignore this term as well.  Indeed, $\b > 2$ whenever
\[
\frac{3}{p} + \frac{2}{q} > 1.
\]
The negation of this inequality is precisely the Prodi-Serrin condition.  Therefore we may assume without loss of generality that $\b >2$, so that
\[
\frac{C}{r^2} \int_{Q_{2r}} |u|^2 \dx\dtau\le \frac{C}{r^\b} \int_{Q_{2r}} |u|^2 \dx\dtau.
\]
The right side of this inequality is the same quantity we use to bound the term $\int |u|^2 |\p_t \phi_r|\dx\dtau$ that appears in \eqref{localei}; therefore it is clear that the addition of the viscous term can cause no trouble.  That is, \eqref{e:Apq} holds whenever the Prodi-Serrin condition fails, while it is obsolete whenever the Prodi-Serrin condition holds.

\begin{REMARK}
We mention one other extension of Proposition \ref{p:A} before moving on, which is applicable only to the NSE.  We can obtain a condition similar to \eqref{e:Apq} under the assumption $u\in L^q L^p$ for some pairs $(p,q)$ with $p<3$, simply by interpolation with the enstrophy space $L^2 H^1$. In particular, this is possible when
\[
\frac9p + \frac5q \le 4, \quad 2<p<3.
\]
If $(p,q)$ satisfies this condition and $u\in L^q L^p$, then $u$ also belongs to $L^a L^3$, where
\[
a = \frac{\frac6p - 1}{ \frac3p + \frac1q - 1}
\]
We can apply Proposition \ref{p:A} to $u\in L^a L^3$, yielding
\[
\sup_{-r^{\widetilde{\a}} < t < 0, \,x_0 \in K} \int_{B_r(x_0)} |u(x,t)|^2\,dx \le C_0\,r^{\widetilde{\b}},
\]
where
\[
\widetilde{\a} = \frac{2(\frac6p - 1)}{\frac3p - \frac1q},
\quad
\widetilde{\b} = \frac{ 4 - \frac9p - \frac5q}{\frac3p - \frac1q}.
\]
\end{REMARK}

When $\|u(t)\|_{L^\infty} \le c_0 |t|^{-1/q}$ (and we use the corresponding time-independent test function $\s = \phi_r=\phi$ from Section \ref{s:typeI}), we may estimate the second term of \eqref{NSEterms} as follows:
\begin{align*}
\int_{t_0}^t \int |u|^2 \Delta \phi\dx\dtau
& \le \left[ \int_{t_0}^t \int_{B_r} |u|^3 \dx\dtau \right]^{\frac23} \cdot \frac{C}{r^2} \cdot [|t_0|r^3]^{\frac13} \\
& \le \left( \frac{|t_0|}{r} \right)^{\frac13} \left[ \frac{C}{r} \int_{t_0}^t f(\tau) E(\tau, 2r) \dtau \right]^{\frac23}\\
& \le C|t_0|r^{-1} + \frac{C}{r} \int_{t_0}^t f(\tau) E(\tau, 2r) \dtau.
\end{align*}
The second term can be absorbed into a term already existing in our energy estimates.  We claim that running the first term through the iteration scheme yields a quantity which can be bounded above by $C|t_0|r^{-1} e^{C_1}$, which is of the same order as $r^{q'-1}$.  Note that this is at least the required order $r^{3-\frac{2q'}{q}}$ whenever $q\le 2$, therefore Proposition \ref{p:typeI} holds for the Navier-Stokes equations when $n=3$ and $5/3\le q\le 2$.  On the other hand, the conclusion is trivial whenever $q>2$, by the Prodi-Serrin criterion.

We now sketch the argument needed to substantiate our claim regarding the term $C|t_0|r^{-1}$.  By making straightforward adjustments to the proof of Lemma \ref{l:step1}, we can write
\begin{equation}
E(t,r)\le r^3 f(s)^2 + \frac{C|s|}{r}+ \frac{C_0}{r}\int_s^t f(\tau) \sum_{j\in \N} 2^{-j} E_j(\tau, r) \dtau,
\quad \quad
-1\le s\le t <0,
\end{equation}
together with its rescaled version
\begin{equation}
E_k(t,r)\le r^3 f(s)^2 + \frac{C|s|}{2^{k(n+1)} r}+ \frac{C_0}{r}\int_s^t f(\tau) \sum_{j\in \N} 2^{-j} E_j(\tau, r) \dtau,
\quad \quad
-1\le s\le t <0,
\end{equation}
the analog of \eqref{e:Ekrescale}. Setting $s=t_0$, we have the analogs of \eqref{e:Et} and \eqref{e:Et1}:
\begin{equation}
E(t,r) = E_0(t,r) \le r^3 f(t_0)^2 + \frac{C|t_0|}{r} + \frac{C_0}{r}\int_{t_0}^t f(t_1) \sum_{k_1\in \N} 2^{-k_1} E_{k_1}(t_1,r) \dt_1,
\end{equation}
\begin{equation}
E_{k_1}(t_1,r) \le r^3  f(t_0)^2 + \frac{C|t_0|}{r} + \frac{C_0}{r} \int_{t_0}^{t_1} f(t_2) \sum_{k_2\in \N} 2^{-k_2} E_{k_2}(t_2,r)\dt_2.
\end{equation}
So our second iterative step becomes
\begin{align*}
E(t,r)
%& \le \left[ r^n f(t_0)^2+ \frac{C|t_0|}{r}\right]\left[ 1  + \frac{C_0}{r}\int_{t_0}^t f(t_1)  \,dt_1\right]  + \frac{C_0}{r}\int_{t_0}^t f(t_1) \sum_{k_1\in \N} 2^{-k_1} \frac{C_0}{r} \int_{t_0}^{t_1} f(t_2) \sum_{k_2\in \N} 2^{-k_2} E_{k_2}(t_2, r)\,dt_2\,dt_1 \\
& \le \left[ r^3 f(t_0)^2+ \frac{C|t_0|}{r}\right]\left[ 1 + C_1 \right] + \frac{C_0^2}{r^2} \int_{t_0}^t f(t_1) \int_{t_0}^{t_1} f(t_2) \sum_{k_1,k_2\in \N} 2^{-(k_1 + k_2)} E_{k_2}(t_2, r)\dt_2\dt_1,
\end{align*}
whereas our $M$th step yields
\begin{align*}
E(t,r) & \le \left[ r^3 f(t_0)^2 + \frac{C|t_0|}{r} \right] \sum_{j=0}^{M-1} \frac{C_1^j}{j!} \\
& \hspace{10 mm}+ \frac{C_0^M}{r^M} \int_{t_0}^t f(t_1) \cdots \int_{t_0}^{t_{M-1}} f(t_M) \sum_{k_1,\ldots,k_M\in \N} \frac{E_{k_M}(t_M, r)}{2^{k_1 + \cdots + k_M}} \dt_M\cdots\dt_1.
\end{align*}
Bounding the two sums and the nested integrals as before, then taking $M\to \infty$, we obtain the bound
\[
E(t,r) \le [r^3 f(t_0)^2 + C|t_0|r^{-1}]e^{C_1}\le Cr^{3 - \frac{2}{q-1}},
\]
justifying our claim.  We pause to record this as a Proposition:
\begin{proposition}
Propositions \ref{p:A} and \ref{p:typeI} remain valid for solutions of 3D Navier-Stokes equation, where $0$ is the first time of blowup.
\end{proposition}

We are now in a position to prove Theorem \ref{t:typeI}.

\begin{proof}
Recall (c.f. Remark \ref{r:Morrey}) that Proposition \ref{p:typeI} can be reframed as the implication 
\[
\|u(t)\|_{L^\infty(\R^n)} \le \frac{c_0}{|t|^{1/q}} 
\implies 
u\in L^\infty \cM^{2, n - \frac{2}{q-1}}.
\]
In the Type-I case for the $3$-dimensional Navier-Stokes equations, we have $q=2$, $n=3$, so that the above becomes 
\[
\|u(t)\|_{L^\infty(\R^3)} \le \frac{c_0}{\sqrt{t}} 
\implies 
u\in L^\infty \cM^{2,1}.
\]
Before proceeding with the proof, we note that this is the implication ``Type-I in time implies Type-I in space'' alluded to in the Introduction.  By ``Type-I in space,'' we mean we mean a blowup which occurs under control of a scaling invariant norm in space---in this case the Morrey space $\cM^{2,1}$ with integrability $2$ and rate $1$.  So it remains to show that the Type-I in space condition implies energy equality.  We argue as follows:  Since $\cM^{1,2}$ is invariant under shifts $f\mapsto f(\cdot - x_0)$ and the rescaling $f(x)\mapsto \l f(\l x)$, we have by Cannone's Theorem \cite{Cannone} that $u\in L^\infty B^{-1}_{\infty, \infty}$.  Consequently, interpolation with the enstrophy space $L^2 H^1 = L^2 B^{1}_{2,2}$ puts the solution into the Onsager-critical class $L^3 B^{1/3}_{3,3}\subset \OR$, from which we conclude energy equality.   
\end{proof}

\section{Concentration Dimension of the Energy Measure}
\label{s:ConcDim}

As explained in Section~\ref{s:dE}, the results above directly imply a lower bound on $D$.  For example, if $u$ belongs to $L^q L^p$ and $(p,q)$ satisfies \eqref{e:pq}, $p\ge 3$, and $q<\infty$, then we have
\begin{equation}
\label{e:Dlwrold}
D \ge \b = \frac{q}{q-1}\left( n - \frac{2n}{p} - \frac{2+n}{q} \right). %= n - \frac{\frac np + \frac2q}{1-\frac1q}.
\end{equation}
It turns out for such pairs $(p,q)$ such that also $p<\infty$, one can obtain a sharper bound by exploiting the local energy inequality directly for the entire cover of a concentration set. The method has been used in the context of Navier-Stokes system in \cite{LS2016b} to obtain dimension-dependent conditions that guarantee energy equality for Leray-Hopf solutions.  We adapt this method to the $n$-dimensional Euler equations and give a refinement in terms of the energy measure.  We also treat this refinement in the case of the 3-dimensional Navier-Stokes equations.  Because we are after a slightly stronger conclusion than in \cite{LS2016b}, we have to adjust the proof at many steps; therefore we repeat most of the argument here.

In what follows, we make use of the following alternate form of the local energy equality in terms of the energy measure.  For any $\s\in C_0^\infty(\O\times (-1,0])$, we have
\begin{equation} \label{e:lee}
\int \s(0)\dE = \int_{-1}^0 \int_\O |u|^2 \p_t \s + (|u|^2 + 2p)u \cdot \n \s - 2\nu(|\n u|^2 \s - u\otimes \n \s:\n u)\dx \dtau,
\end{equation}
where as usual we understand that $\nu = 0$ for the Euler Equations.
Notice that we have killed off the initial data by requiring support in $\O\times (-1,0]$ rather than $\O\times [-1,0]$. By approximation, \eqref{e:lee} holds for all $\s \in \Lip_0((-1,0] \times \O)$.

\subsection{Euler Equations}

\begin{THEOREM}
	\label{t:oldmethod}
	Suppose $u\in L^q L^p(\O)$ for some $(p,q)$ satisfying \eqref{e:pq}, and suppose $d\ge 0$ satisfies
\begin{equation}\label{e:dpq}
\begin{split}
d  \le n - \frac{\frac2q}{1-\frac{2}{p}-\frac1q},
& \quad q\le p\le \infty, \;q<\infty \\
d  <  n - \frac{\frac2q}{1-\frac{2}{p}-\frac1q},
& \quad 3\le p < q <\infty.\\
d  < n,  & \quad 2<p\le \infty=q.
\end{split}
\end{equation}
Then $\cE(S) = 0$ for every $S\subset \O$ with finite $d$-dimensional Hausdorff measure.  In particular, if $\dim_{\cH}(\Sigma_{ons})$ satisfies \eqref{e:dpq}, then $u$ satisfies the local energy equality on $[-1,0]$.  Regardless of the size of $\Sigma_{ons}$, the right side of \eqref{e:dpq} gives a lower bound on the concentration dimension $D$.  Similarly, if $\dim_\cH(\Sigma)$ satisfies \eqref{e:dpq}, then $D=n$.
\end{THEOREM}

\begin{figure}[h]
		\includegraphics{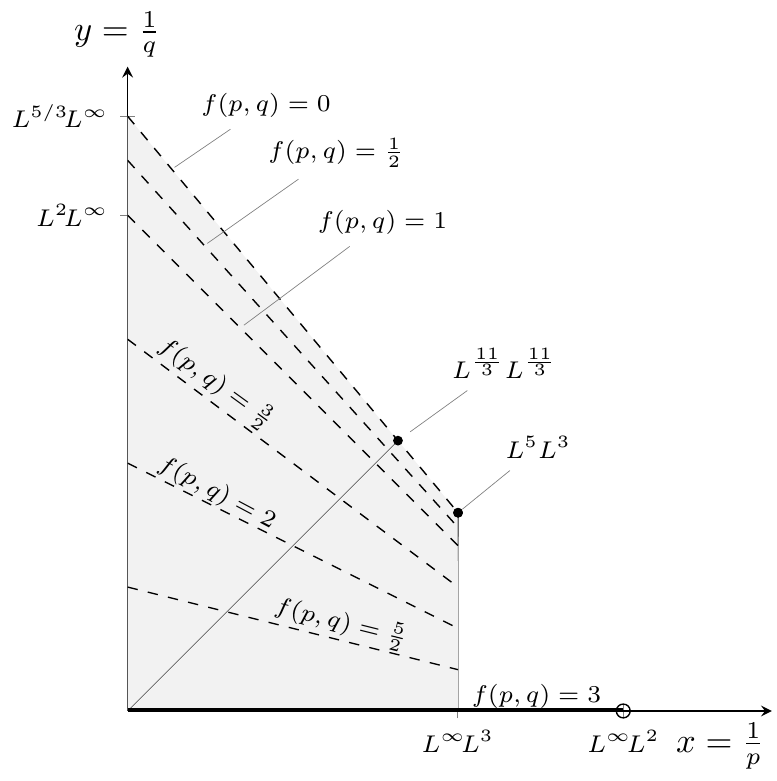}
		\caption{ Define a function $f(p,q)$ by the right side of \eqref{e:dpq}.  The dashed lines depict level curves of $f(p,q)$ in the case $n=3$ (specified for the sake of definiteness).  \label{fig:EE} }
\end{figure}

\begin{REMARK}
Notice that
\[
n - \frac{\frac2q}{1-\frac{2}{p}-\frac1q}
= \frac{ n - \frac{2n}{p} - \frac{2 + n}{q}}{1 - \frac2p - \frac1q}
\ge \frac{ n - \frac{2n}{p} - \frac{2 + n}{q}}{1 - \frac1q} = \b,
\]
with equality precisely when $p=\infty$.  So the lower bound on $D$ given by the present theorem is indeed better than \eqref{e:Dlwrold} except when $p=\infty$, when it is the same.
\end{REMARK}

\begin{proof}
	The statement regarding the concentration set $\Sigma$ is a direct consequence of Corollary \ref{c:newD}. Next, recall from Section \ref{s:dE} that the two conditions $|\Sigma_{ons}|=0$ and $\cE(\Sigma_{ons})=0$ together imply energy equality.  Of course, if $\dim_\cH(\Sigma_{ons})<n$ (which occurs whenever $\dim_\cH(\Sigma_{ons})$ satisfies \eqref{e:dpq}), then $|\Sigma_{ons}|=0$ trivially and we need only prove $\cE(\Sigma_{ons})=0$ in order to conclude energy equality.  Therefore the conclusion holds trivially for $u\in L^\infty L^p(\O)$, $p>2$, since we have already proven $D=n$ in this case (see Corollary \ref{c:Lp}).
	
	It remains to show that $\cE(S) = 0$ whenever $\dim_\cH(S)$ satisfies \eqref{e:dpq}.  Let us first reduce to the case $p,q\in [3,\infty)$.  This can fail for three reasons: $q<3$, $p=\infty$, or $q=\infty$.  We have already dealt with the last case; the other two are covered by the following interpolation argument. Suppose $q<p\le \infty$ and put $r = 2 + q - \frac{2q}{p}$.  Then $(r,r)$ satisfies $\eqref{e:pq}$, $r\in [3,\infty)$, and $u\in L^r L^r$, as it lies along the line segment joining $L^q L^p$ with $L^\infty L^2$. (That is, $(\frac1r, \frac1r)$ lies between $(\frac1p, \frac1q)$ and $(\frac12, 0)$ on the line $2px + q(p-2)y = p$.) Furthermore, it is easy to check that
	\[
	n - \frac{2}{r - 1 - \frac{2r}{r}} = n - \frac{2}{q - 1 - \frac{2q}{p}},
	\]
	and therefore that
	\[
	d \le n - \frac{2}{r - 1 - \frac{2r}{r}},
	\]
	so that $\cE(S)=0$, as desired.  For the remainder of the proof, we assume that $p,q\in [3,\infty)$.
	
	Choose $\d\in (0,\e/3)$, then choose $x_i \in \O_{\e}$, $r_i\in (0,\d)$ for all $i$, such that $S\subset \bigcup_i B_{r_i}(x_i)$ and $\sum_{i=1}^\infty r^d_i \lesssim \cH_d(S)+1$.  Denote $I_i = (-2r_i^{\a},0)$ (where $\a$ is determined below).
	%; let $Q_i$ denote the cylinder $Q_i = B_{r_i}(x_i)\times (-r_i^\a,0)$, and put $Q = \bigcup_i Q_i$, $I = \bigcup_i I_i$.  
	Let $\psi(s)$ be the usual (symmetric, radially decreasing) cut-off function on the line with $\psi(s) = 1$  on $|s|<1.1$ and $\psi(s)$ vanishing on $|s|>1.9$. Let $\phi_i(x,t) = \psi(|x-x_i|/r_i) \psi(t/r_i^{\a})$. Define
	\[
	\phi^N = \sup_{1\le i\le N} \phi_i,
	\quad \quad
	\phi = \sup_{i\in \N} \phi_i.
	\]
	Then each $\phi^N$ is continuous with support in $\O_\e\times (-1,0]$, $0\le \phi^N \le 1$, and $\phi^N$ increases pointwise to $\phi$, which is identically $1$ on $S\times \{0\}$. So
	\[
	\cE(S) \le \int \phi(0) \dE = \lim_{N\to \infty} \int \phi^N(0)\dE,
	\]
	by monotone convergence.  Furthermore, each $\phi^N$ is differentiable a.e., with
	\begin{equation}
	\label{eq:Evansthm}
	|\p \phi^N(x,t) | \leq \sup_{1\le i\le N} |\p \phi_i(x,t)|, \, \text{ a.e., see \cite[Theorem 4.13]{Evans}}.
	\end{equation}
	(In fact, we even have $|\p \phi(x,t) | \leq \sup_{i\in \N} |\p \phi_i(x,t)|$, though we don't use it.) Therefore, an approximation argument shows that we can put $\phi^N$ in the local energy equality:
	\[
	\lim_{t\to 0} \int_\O |u(t)|^2 \phi^N(t)\dx = \int_{-1}^0 \int_{\O} |u|^2 \p_t \phi^N + (|u(t)|^2 + 2p)u\cdot \n \phi^N \dx \dtau.
	\]
	
	Putting all this together, we obtain
	\begin{equation}
	\label{e:ESphiN}
	\cE(S) \le \lim_{N\to \infty} \int_{-1}^0 \int_{\O} |u|^2 \p_t \phi^N + (|u|^2 + 2p)u\cdot \n \phi^N \dx \dtau.
	\end{equation}
	For $d$ small enough, we will obtain uniform bounds on
	\[
	C_N = \int_{-1}^0 \int_\O |u|^2 \p_t \phi^N\dx \dtau,
	\quad
	D_N + P_N = \int_{-1}^0 \int_\O (|u|^2 +2p)u\cdot \n \phi^N \dx \dtau.
	\]
	Using H\"older's inequality and \eqref{eq:Evansthm}, we have
	\begin{align*}
	C_N & \leq C\|u\|_{L^q(I,L^p(\O))}^2 \left( \int_{-1}^0 \left( \sum_{i=1}^N r_i^{- \frac{\a p}{p-2} + n}\chi_{I_i}(t) \right)^{\frac{p-2}{p}\frac{q}{q-2}} \dt \right)^{\frac{q-2}{q}} \\
	D_N + P_N & \leq C\|u\|_{L^q(I,L^p(\O))}^3  \left(  \int_{-1}^0 \left( \sum_{i=1}^N r_i^{-\frac{p}{p-3} +n} \chi_{I_i}(t) \right)^{\frac{p-3}{p} \frac{q}{q-3}} \dt \right)^{\frac{q-3}{q}}
	\end{align*}
	Note that to bound $P_N$, we have also used boundedness of the Riesz transforms on $L^{p/2}$ (recall that $p\in [3,\infty)$).  That is, we use the bound $\|p\|_{L^{p/2}} \le C\|u\|_{L^p}^2$ before exhausting the remaining integrability on $\n \phi^N$.  The following lemma allows us to bound the quantities $C_N$, $D_N+P_N$ for small enough $d$:
	
	\begin{LEMMA}[\cite{LS2016b}]\label{l:conv} Let $d$, $\d$, $r_i$, $I_i$ be as above, and let $\s,s$ be positive numbers.  Suppose the sum $H = \sum_i r_i^d$ is finite. Then the inequality
		\emph{\begin{equation}
		\label{lemmaineq}
		\int \left( \sum_i r_i^{-\s}\chi_{I_i}(t) \right)^s\dt  \lesssim H^s
		\end{equation}}
		holds whenever $s\ge 1$ and $d\le \frac{\a}{s}-\s$, or $s<1$ and $d < \frac{\a}{s}-\s$. When $d = 0$, the above holds (trivially) under the non-strict assumption $0 \leq \frac{\a}{s}-\s$.
	\end{LEMMA}
	
	Following \cite{LS2016b}, we translate the hypotheses and conclusion of the lemma into statements involving $p$, $q$, $\a$, and $d$; then we optimize in $\a$. When dealing with $C_N$, we set $\s = \frac{\a p}{p-2} - n$ and $s = \frac{p-2}{p}\frac{q}{q-2}$.  Denoting $H_N = \sum_{i=1}^N r_i^d$, we conclude that
	\begin{equation}
	\label{e:CN}
	C_N \le \|u\|_{L^q(I,L^p(\O))}^2 H_N^{1 - \frac2p}
	\quad \text{ whenever } \quad
	\left\{
	\arraycolsep=1.4pt\def\arraystretch{2}
	\begin{array}{lcl}
	d\le n - \frac{\frac2q \a}{1 - \frac2p}, & & 2\le q\le p\le \infty,\\
	d < n - \frac{\frac2q \a}{1 - \frac2p}, & & 2\le p< q\le \infty.
	\end{array}
	\right.
	\end{equation}
	Dealing with $D_N + P_N$ is essentially the same: we set $\s = \frac{p}{p-3} - n$ and $s = \frac{p-3}{p}\frac{q}{q-3}$ and apply the lemma to conclude that
	\begin{equation}
	\label{e:DNPN}
	D_N + P_N \le \|u\|_{L^q(I,L^p(\O))}^3 H_N^{1 - \frac3p}
	\quad \text{ whenever } \quad
	\left\{
	\arraycolsep=1.4pt\def\arraystretch{2}
	\begin{array}{lcl}
	d\le n - \frac{1 - \a(1 - \frac3q)}{1 - \frac3p}, & & 3\le q\le p<\infty,\\
	d < n - \frac{1 - \a(1 - \frac3q)}{1 - \frac3p}, & & 3\le p< q\le \infty.
	\end{array}
	\right.
	\end{equation}
	In light of the bound $H_N \le \cH_d(S)+1$, which is uniform in both $N$ and $\d$, we have
	\[
	\cE(S) \le \|u\|_{L^q(I,L^p(\O))}^2 (\cH_d(S)+1)^{1 - \frac2p} + \|u\|_{L^q(I,L^p(\O))}^3 (\cH_d(S)+1)^{1 - \frac3p},
	\]
	by \eqref{e:ESphiN}, whenever the conditions on $d$ from \eqref{e:CN}, \eqref{e:DNPN} are satisfied for some $\a$.  (Note that, while the estimates on $C_N$ and $D_N + P_N$ are valid for the ranges of $p$ and $q$ stated above, we continue to work under the assumption that $p$ and $q$ both lie in the range $p,q\in [3,\infty)$.)  Since $|I|\to 0$ as $\d\to 0$, we have  $\|u\|_{L^q(I, L^p(\O))}\to 0$ as well (as $q<\infty$), and therefore $\cE(S) = 0$.  The choice of $\a$ which maximizes
	\[
	\min\left\{n - \frac{1 - \a(1 - \frac3q)}{1 - \frac3p},\; n - \frac{\frac2q \a}{1 - \frac2p}
	\right\}
	\]
	(and therefore gives the optimal range for $d$) is given by
	\begin{equation}
	\label{e:alphaEE}
	\a = \frac{ 1 - \frac2p }{1 - \frac2p - \frac1q}.
	\end{equation}
	Substituting this value of $\a$ into the conditions on $d$ derived above, we conclude that $\cE(S) = 0$ whenever
	\begin{equation}
	\begin{split}
	d  \le n - \frac{\frac2q}{1-\frac{2}{p}-\frac1q},
	& \quad 3\le q\le p<\infty \\
	d  <  n - \frac{\frac2q}{1-\frac{2}{p}-\frac1q},
	& \quad 3\le p < q<\infty.
	\end{split}
	\end{equation}
	This completes the proof.
\end{proof}

\subsection{Navier-Stokes Equations}
\label{s:NSEConcdim}

In the case of the NSE, the optimal condition on $d$ analogous to \eqref{e:dpq} breaks into many different parts, depending on $p$ and $q$.  To streamline the statement of the theorem, let us introduce notation for the various regions involved.
\begin{equation}
\begin{split}
\text{I} & := \bigg\{(p,q): p\ge q,\;\; \frac1p + \frac1q > \frac12,\;\; \frac6p + \frac5q \le 3\bigg\}, \\
\text{II} & :=\bigg\{(p,q): 3\le p < q,\;\;\frac1p + \frac1q \ge \frac12,\;\; \frac6p + \frac5q \le 3\bigg\} \\
\text{III} & := \bigg\{(p,q): 2< p < 3,\;\; \bigg(\frac12 - \frac1p\bigg)\bigg(2-\frac3p\bigg) \le \frac1q \le \bigg(\frac12 - \frac1p\bigg)\bigg(2-\frac3p\bigg)\bigg(\frac76 - \frac1p\bigg)^{-1}\bigg\}\\
\text{IV} & :=\bigg\{(p,q):\frac1p + \frac1q \le \frac12,\;\;\frac3p + \frac1q \le 1\bigg\}, \\
\text{V} & :=\bigg\{(p,q):\frac1p + \frac1q < \frac12,\;\;\frac1q<\bigg(\frac12 - \frac1p\bigg)\bigg(2-\frac3p\bigg)\bigg\}\big\backslash \text{IV} \\
\end{split}
\end{equation}

Let us introduce also a piecewise function defined on these regions, which will serve as a sort of threshold dimension in what follows:
\begin{equation}
f(p,q):=\left\{
		\arraycolsep=1.4pt\def\arraystretch{2}
		\begin{array}{lcl}
			3 - \frac{\frac2q}{1-\frac{2}{p}-\frac1q},
			& & (p,q)\in \text{I}\cup \text{II}  \\
			\vspace{2 mm}
			3 - \frac{\frac2q (\frac6p-1)}{(2 - \frac3p - \frac3q)(1-\frac{2}{p})},
			& & (p,q)\in \text{III}, \\
			3, & & (p,q)\in \text{IV}\cup \text{V}  \\
		\end{array}
		\right.
\end{equation}
\begin{comment}
\begin{figure}[h]
	\begin{minipage}{0.45\linewidth}
		\includegraphics{./C&D_regions_fig_v2.pdf}
		\vspace{-8 mm}		
		\caption{The regions of $(p, q)$-space involved in the statement of Theorem~\ref{t:oldNSE}. The boundary segments are dotted according to which of two adjoining regions contains the segment in question.   
\label{fig:regions} }
	\end{minipage}
	\begin{minipage}{0.45\linewidth}
		\includegraphics[width=\textwidth]{./C&D_3dplot_v3.png}
		\vspace{-8 mm}
		\caption{A 3D plot of the graph of $f(p,q)$ 		\label{fig:3d} }
	\end{minipage}
\end{figure}
\end{comment}

\begin{figure}[h]
\includegraphics[width = 0.5\textwidth]{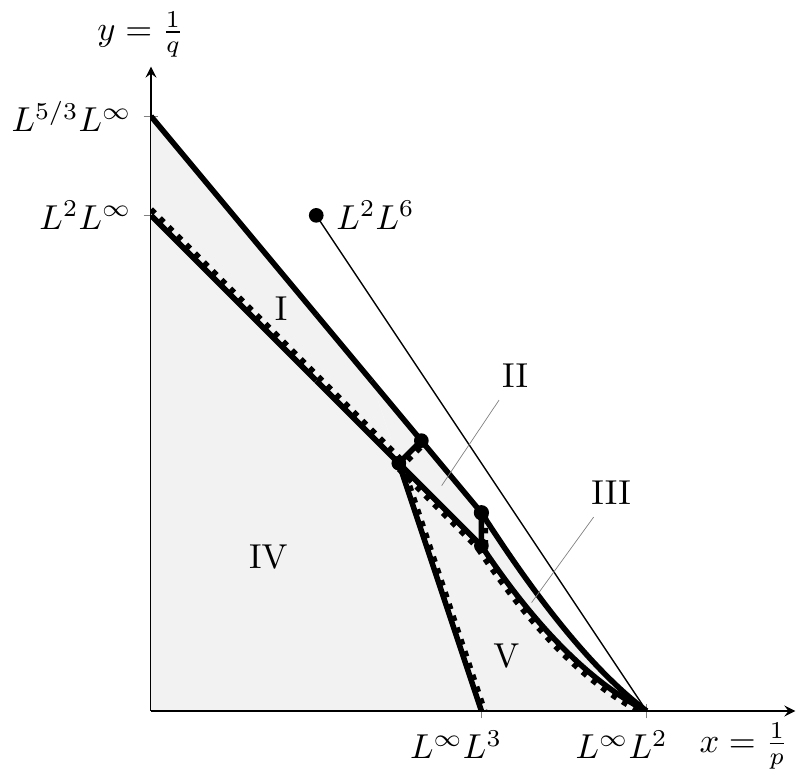}
\caption{The regions of $(p, q)$-space involved in the statement of Theorem~\ref{t:oldNSE}. To highlight values of $f(p,q)$ along  jump discontinuities the boundary segments are dotted according to which of two adjoining regions contains the segment in question.  For example, Region IV contains the segment joining $L^2 L^\infty$ with $L^4 L^4$ along which $f = 3$, but Region II contains the segment joining $L^4 L^4$ with $L^6 L^3$.  
\label{fig:regions} }
\end{figure}

\begin{figure}[h]
\includegraphics[width=0.5\textwidth]{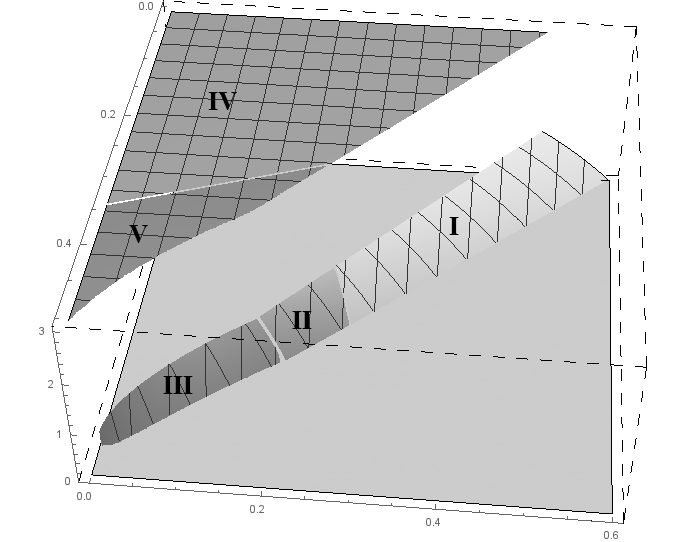}
%\vspace{-8 mm}
\caption{A 3D plot of the graph of $f(p,q)$ 		
\label{fig:3d} }
\end{figure}

\begin{THEOREM}
	\label{t:oldNSE}
	Suppose $u\in L^q L^p(\O)$ for some $(p,q)$ satisfying \eqref{e:pq}, and suppose $d\ge 0$ satisfies
	\begin{equation}
	\label{e:dpqNSE}
	\begin{split}
	d \le f(p,q), & \quad (p,q)\in \emph{I} \\
	d < f(p,q), & \quad (p,q)\in \emph{II} \cup \emph{III} \cup \emph{IV} \cup \emph{V}.
	\end{split}
	\end{equation}
Then $\cE(S) = 0$ for every $S\subset \O$ with finite $d$-dimensional Hausdorff measure.  In particular, if $\dim_{\cH}(\Sigma_{ons})$ satisfies \eqref{e:dpqNSE}, then $u$ satisfies the local energy equality.  Regardless of the size of $\Sigma_{ons}$, the right side of \eqref{e:dpqNSE} gives a lower bound on the concentration dimension $D$.  Similarly, if $\dim_\cH(\Sigma)$ satisfies \eqref{e:dpqNSE}, then $D=3$.
\end{THEOREM}

\begin{REMARK}
For $(p,q)\in \text{III}$, the formula defining $f(p,q)$ can be written as a deviation from the formula in the neighboring region II:
\[
3 - \frac{\frac2q (\frac6p-1)}{(2 - \frac3p - \frac3q)(1-\frac{2}{p})} = 3 - \frac{\frac2q}{1 - \frac2p - \frac1q} - \frac{ (\frac3p-1)(3 - \frac6p - \frac4q)}{(1 - \frac2p)(1 - \frac2p - \frac1q)(2 - \frac3p - \frac3q)}.
\]
\begin{comment}
(c.f. Region III) or alternatively,
\[
3 - \frac{\frac2q (\frac6p-1)}{(2 - \frac3p - \frac3q)(1-\frac{2}{p})} = 3 - \frac{\frac4q}{1 - \frac2p} - \frac{\frac13(\frac{2}{p} + \frac1q - \frac56)}{(1 - \frac2p)(2 - \frac3p - \frac3q)}	
\]
(c.f. Region IV).
\end{comment}
\end{REMARK}

\begin{proof}
The claims regarding $\Sigma$ and $\Sigma_{ons}$ follow by the same reasoning as in the proof of Theorem \ref{t:oldmethod}.  Also, the local energy equality for the space $L^4 L^4$ gives the result in region IV.  To treat the remaining regions, let us use the same setup and test function as in the proof of Theorem \ref{t:oldmethod}, assuming $n=3$.  Since $\nu>0$ now, we do have to consider the two additional terms 
\[
E_N = \int_{-1}^0 \int_\O |\n u|^2 \phi^N\dx\dtau,
\quad \quad 
F_N = \int_{-1}^0 \int_\O u\otimes \n \phi^N : \n u \dx\dtau.
\]
Taking $N\to \infty$ and then $\d\to 0$, it is clear that $E_N$ vanishes in the limit regardless of $(p,q)$.  It turns out that $F_N$ is also never limiting with respect to the best possible value of $D$, but it takes a bit of work to see this.  

We treat the remaining four regions in turn as follows:
\begin{itemize}
	\item In Regions I and II, we reuse the bounds \eqref{e:CN} and \eqref{e:DNPN} (with $n=3$), and we show that the analogous bounds for $F_N$ are strictly better in these two regions.  Strictly speaking, this argument only works for $q\ge 3$, but we can use the same logic as in the previous proof to cover the missing region $\text{I}\cap \{q<3\}$.  
	\item In Region III, we reuse \eqref{e:CN} once again, but \eqref{e:DNPN} is no longer valid for any pair $(p,q)$ under consideration.  We give a replacement bound using the enstrophy, which is valid for $2<p<3$, then we optimize as in the previous theorem.  
	\item In Region V, we use a sort of bootstrap argument.  First, we construct a function $g(p,q)$ defined on Region V such that $\cE(S)=0$ whenever $\dim_\cH(S)<g(p,q)$.  Then, we show that $g(p,q)>1$ everywhere on V.  By the discussion in Section \ref{s:suitable}, we know that $\dim_{\cH}(\Sigma_{ons})\le 1<g(p,q)$, and therefore that $\cE(\Sigma_{ons})=0$.  But this implies that $\dE = |u(0)|^2\dx$ and that the local energy equality holds for $t=0$.  This obviously implies $D=3$, which is the desired conclusion for Region V.
\end{itemize}

\textbf{Step 1: Regions I and II.} In accordance with the outline above, we assume without loss of generality that $q\ge 3$. Estimating
	\[
	F_N \leq C\|u\|_{L^q(I,L^p(\O))}^2 \|\n u\|_{L^2 L^2} \left( \int_{-1}^0 \left( \sum_{i=1}^N r_i^{- \frac{ 2p}{p-2} + 3}\chi_{I_i}(t) \right)^{\frac{p-2}{p}\frac{q}{q-2}} \dt \right)^{\frac{q-2}{2q}}
	\]
	and applying Lemma \ref{l:conv}, we conclude that
	\begin{equation}
	\label{e:FN}
	F_N \le \|u\|_{L^q(I,L^p(\O))} \|\n u\|_{L^2 L^2} H_N^{\frac12 - \frac1p}
	\quad \text{ whenever } \quad
	\left\{
	\arraycolsep=1.4pt\def\arraystretch{2.2}
	\begin{array}{lr}
	d\le 3 - \frac{\frac2q \a - (\a - 2)}{1 - \frac2p}, & 2\le q\le p \le \infty,\\
	d < 3 - \frac{\frac2q \a - (\a - 2)}{1 - \frac2p}, & 2\le p < q \le \infty.
	\end{array}
	\right.
	\end{equation}
	Comparison with \eqref{e:CN} shows that our condition on $F_N$ is superfluous if $\a\ge 2$, which is satisfied by \eqref{e:alphaEE} exactly when $\frac1p + \frac1q \ge \frac12$.  But $\frac1p + \frac1q \ge \frac12$ holds for all $(p,q)\in \text{I}\cup \text{II}$.  Therefore we can ignore $F_N$ in Regions I and II and read off the relevant conclusion from the previous theorem in these regions (with $n=3$).  That is, $\cE(S) = 0$ holds if $S$ has finite $d$-dimensional Hausdorff measure, where $d$ satisfies
	\begin{equation}
	\begin{split}
	d \le 3 - \frac{ \frac2q }{1 - \frac2p - \frac1q}, & \quad (p,q)\in \text{I},\\
	d < 3 - \frac{ \frac2q }{1 - \frac2p - \frac1q}, & \quad (p,q)\in \text{II}.
	\end{split}
	\end{equation}
	This takes care of Regions I and II completely.  However, before proceeding we note that our condition cannot be improved by  considering $\a<2$: in this case, the bounds on $d$ from \eqref{e:DNPN} and \eqref{e:FN} can both be improved by increasing $\a$, while the bound on $d$ from \eqref{e:CN} is superfluous.  
	
\textbf{Step 2: Region III.} Region III lies entirely in the range $\{2<p<3\}$, where \eqref{e:DNPN} is not applicable.  Therefore we estimate $D_N + P_N$ differently:
	
	\begin{equation}
	\label{e:DPpl3}
	D_N+P_N\le \|u\|_{L^2 H^1}^{3\b} \|u\|_{L^q L^p}^{3(1-\b)}\left( \int \sup_i r_i^{-\s} \chi_{I_i}(t)\dt\right)^{\frac1\s},
	\end{equation}
	where
	\begin{equation}
	\label{e:defbeta}
	\frac13 = \frac{\b}{6} + \frac{1-\b}{p}\implies
	\b = \frac{\frac6p-2}{\frac6p-1};
	\quad \frac1\s = 1 - \frac{3\b}{2} - \frac{3(1-\b)}{q} = \frac{2 - \frac3p - \frac3q}{\frac6p-1}.
	\end{equation}
	It is shown in \cite{LS2016b} that
	\[
	\int \sup_i r_i^{-\s} \chi_{I_i}(t)\dt \lesssim \sum_j (2^{\a-\s})^{-j},
	\]
	which is bounded whenever $\s < \a$.  Note that this condition is \emph{independent of $d$}, so we formulate it as a bound on $\a$:
	\begin{equation}
	\label{e:DPenstr}
	\a > \frac{ \frac6p - 1 }{2 - \frac3p - \frac3q}.
	\end{equation}
	Reasoning as before, we have control over both $C_N$ and $F_N$ whenever
	\begin{equation}
	\label{e:CNrest}
	d < 3 - \frac{\frac2q \a}{1 - \frac2p}
	\end{equation}
	and $\a\ge 2$.  Now
	\begin{equation}
	\label{e:ag2}
	\frac{ \frac6p - 1 }{2 - \frac3p - \frac3q} > 2 \iff \frac2p + \frac1q > \frac56,
	\end{equation}
	and every pair $(p,q)$ in Region III satisfies $\frac2p + \frac1q > \frac56$.  (This is not difficult to show algebraically, but it is even easier to see geometrically by noting that the line $\frac2p + \frac1q = \frac56$ passes through $L^6 L^3$ and $L^\infty L^{12/5}$.)
	
	Therefore we can substitute \eqref{e:DPenstr} into \eqref{e:CNrest} in order to conclude that $\cE(S) = 0$ whenever $d=\dim_\cH(S)$ satisfies
	\begin{equation}
	\label{e:III}
	d < 3 - \frac{ \frac2q(\frac6p - 1) }{(1 - \frac2p)(2 - \frac3p - \frac3q)}, \quad (p,q)\in \text{III}.
	\end{equation}
	This is the desired conclusion for Region III.
	
\textbf{Step 3: Region V.} The general strategy for dealing with Region V is outlined at the beginning of the proof.  We recall that to complete the proof, it suffices to find a function $g(p,q)>1$ on V such that $\dim_\cH(S)<g(p,q)$ implies that $\cE(S) = 0$.  In order to define this function $g(p,q)$, we first split the region V further into three pieces:
	\begin{equation}
	\text{V}_a:=\text{V}\cap\{p\ge 3\},
	\quad 
	\text{V}_b:=\text{V}\cap\{p<3\}\cap \left\{\frac2p + \frac1q\le \frac56\right\},
	\quad 
	\text{V}_c:=\text{V}\cap \left\{\frac2p + \frac1q > \frac56 \right\}.
	\end{equation}
	In Region $\text{V}_c$, we can reason as in Step 2 and define $g(p,q)$ by the right side of \eqref{e:III}.  Furthermore, 
	\begin{align*}
	g(p,q) = 3 - \frac{ \frac2q(\frac6p - 1) }{(1 - \frac2p)(2 - \frac3p - \frac3q)}
	& = 3 - \frac{2(1 - \frac2p)(2 - \frac3p - \frac3q) - 4((2 - \frac3p)(\frac12 - \frac1p) - \frac1q)}{(1 - \frac2p)(2 - \frac3p - \frac3q)} \\
	& = 1 + \frac{4[(2 - \frac3p)(\frac12 - \frac1p) - \frac1q]}{(1 - \frac2p)(2 - \frac3p - \frac3q)}>1,
	\end{align*}
	since $\frac1q<(2 - \frac3p)(\frac12 - \frac1p)$ for $(p,q)\in \text{V}$.  	
		
	In the Regions $\text{V}_a$ and $\text{V}_b$, we set $\a=2$ and note that the restrictions on $d$ due to \eqref{e:CN} and \eqref{e:FN} coincide in this case.  On the other hand, the restriction due to $D_N + P_N$ becomes superfluous here.  This is easy to see for $(p,q)\in \text{V}_b$, since \eqref{e:DPenstr} is satisfied in this region by \eqref{e:ag2}.  For $(p,q)\in \text{V}_a$, one can also compare \eqref{e:CN} and \eqref{e:DNPN} directly, but the following argument is perhaps more insightful: 	Notice that, as we increase $\a$, the requirements on $d$ become more stringent for \eqref{e:CN} and less stringent for \eqref{e:DNPN}.  The two conditions coincide when $\a$ is given by \eqref{e:alphaEE}.  As discussed in Step 1, this value of $\a$ is less than $2$ if $\frac1p + \frac1q<\frac12$, which is satisfied for all pairs $(p,q)\in \text{V}$.  Therefore, \eqref{e:DNPN} is superfluous when $(p,q)\in \text{V}_a$ and $\a=2$.  We therefore define 
	\[
	g(p,q) = 3 - \frac{ \frac4q }{1 - \frac2p}, 
	\quad (p,q)\in \text{V}_a\cup \text{V}_b.
	\]
	Since $\frac2p + \frac2q <1$ in V, we have 
	\begin{equation}
	\label{e:Vsuitable}
	3 - \frac{ \frac4q }{1 - \frac2p} = \frac{2(1 - \frac2p - \frac2q) + (1 - \frac2p)}{1 - \frac2p} >1,
	\quad (p,q)\in \text{V},
	\end{equation}
	and therefore $g(p,q)>1$ in $\text{V}_a\cup \text{V}_b$, and therefore in all of V.  This completes the proof.	
\end{proof}

\begin{REMARK}
	In proving Theorem \ref{t:oldNSE}, we also proved that our results were optimal for our method, by showing that we had chosen the best possible time scaling $\a$.  In principle, we could have instead solved various inequalities from \cite{LS2016b} for $d$ and simply adjusted the proof to show that the stronger statement still held.  We write out the optimization argument explicitly for two reasons. First, the present argument gives us an explicit best time-scaling.  Secondly, the perspective here is a bit different than in \cite{LS2016b}.  There, the authors fix $d$ and give conditions on $p$ and $q$ so that energy balance holds, and the form of the conditions on $p$ and $q$ is not the same for all values of $d$.  In contrast, we now fix $p$ and $q$ and give conditions on $d$, which take different forms depending on $p$ and $q$.  Translating between these various conditions would not be much quicker than re-optimizing as we have done here, and would be less motivated.
\end{REMARK}

%\bibliographystyle{plain}
%\bibliography{CD}

\def\cprime{$'$}

\end{document}